\documentclass[journal,twoside,web]{ieeecolor}
\usepackage{generic}
\usepackage{float}
\usepackage{cite}
\usepackage{amsmath,amssymb,amsfonts}
\usepackage{bm}
\usepackage{algorithmic}
\usepackage{graphicx}
\usepackage{algorithm,algorithmic}
\usepackage{hyperref}
\hypersetup{hypertex=true, colorlinks=true, linkcolor=blue, citecolor=blue}
\usepackage{lipsum}
\usepackage{tabularx}
\usepackage{multirow,array}
\usepackage{makecell}
\usepackage{booktabs}
\usepackage{caption,subcaption}
\usepackage{xr}
\usepackage{textcomp}
\def\BibTeX{{\rm B\kern-.05em{\sc i\kern-.025em b}\kern-.08em
		T\kern-.1667em\lower.7ex\hbox{E}\kern-.125emX}}
 \renewcommand{\thesection}{\arabic{section}}
 \renewcommand{\thesubsection}{\arabic{section}.\arabic{subsection}}

  \newtheorem{theorem}{Theorem}
  \newtheorem{definition}{Definition}
  \newtheorem{corollary}{Corollary}
  \newtheorem{lemma}{Lemma}
 \newtheorem{assumption}{Assumption}
\newtheorem{remark}{Remark}
 \newtheorem{proposition}{Proposition}

\begin{document}
\title{Distributed Stochastic Proximal Algorithm on Riemannian Submanifolds for Weakly-convex Functions}
\author{Jishu Zhao, Xi Wang, Jinlong Lei,  \IEEEmembership{Member, IEEE}, and Shixiang Chen,
	\thanks{This paragraph of the first footnote will contain the date on 
		which you submitted your paper for review. This work was funded by the National Natural Science Foundation of
		China under Grant No. 72271187, and the Fundamental Research Funds for the Central
		Universities. }
	\thanks{Jishu Zhao, Jinlong Lei are with the Department of Control Science and Engineering, Tongji University, Shanghai, 201804, China. Lei is also with the State Key Laboratory of Autonomous Intelligent Unmanned Systems, and Frontiers Science Center for Intelligent Autonomous Systems, Ministry of Education, Tongji University, Shanghai, 201804, China. (e-mail: 2111115@tongji.edu.cn, leijinlong@tongji.edu.cn). }
	\thanks{Xi Wang is with the School of Electrical Engineering and Telecommunications, University of New South Wales, Sydney, NSW 2052, Australia (e-mail: xi.wang14@unsw.edu.au).}
	\thanks{Sixiang Chen is with the Key Laboratory of the Ministry of Education for Mathematical Foundations and Applications of Digital Technology, School of Mathematical Sciences, University of Science and Technology of China, Hefei, 230026, China  (e-mail: shxchen@ustc.edu.cn).}
}

\maketitle
\begin{abstract}
	This paper aims to investigate the distributed stochastic optimization problems on compact embedded submanifolds (in the Euclidean space) where the local cost functions are weakly-convex. To address the manifold structure, we propose a distributed Riemannian stochastic proximal algorithm framework by utilizing the retraction and Riemannian consensus protocol, and analyze three specific algorithms: the distributed Riemannian stochastic subgradient, proximal point, and prox-linear algorithms. When the initial points satisfy certain conditions, we show that the iterates generated by this framework converge to a nearly stationary point in expectation while achieving consensus. We further establish the convergence rate of the algorithm framework as $\mathcal{O}(\frac{1+\kappa_g}{\sqrt{k}})$ where $k$ denotes the number of iterations and $\kappa_g$ shows the impact of manifold geometry on the algorithm performance. Finally, numerical experiments are implemented to demonstrate the theoretical results and show the empirical performance.
\end{abstract}

\begin{keywords}
	Distributed stochastic optimization, Riemannian manifolds, Weakly-convex, Retraction-based.
\end{keywords}	

\section{Introduction}
\IEEEPARstart{W}{ith} the increasing need to reduce computational cost and privacy protection in large-scale optimization problems, distributed optimization has attracted much research attention in recent years, where all the agents work cooperatively to solve the problem. It is worth noting that decision variables in some particular applications have manifold structure, e.g., distributed spectral analysis \cite{huang2020communication, kempe2004decentralized}, deep learning in hyperbolic space \cite{Yang2024ICML}, or with orthogonal constraints \cite{vorontsov2017orthogonality}, sparse blind deconvolution \cite{Zhang_2017_CVPR}. Therefore, agents need to obey a nonlinear and nonconvex manifold constraint when solving the problem, which brings challenges and difficulties in designing distributed algorithms. 

In this paper, we consider the following distributed stochastic optimization problems on Riemannian submanifolds
\begin{equation}\label{p1}
\min_{X \in \mathcal{M}} ~f(X) = \frac{1}{N} \sum_{i=1}^{N} f_i(X),~~f_i(X):=\mathbb{E}_{\xi_i}[\tilde{F}_i(X, \xi_i)],
\end{equation}
where $\mathcal{M}$ is a compact embedded submanifold of the $n$-dimensional Euclidean space $ \mathbb{R}^n$, and cost functions $f_i: \mathcal{M} \mapsto \mathbb{R}$ are weakly-convex functions in the ambient Euclidean space $\mathbb{R}^n$. Herein and after, "submanifolds" represents the compact embedded submanifolds unless otherwise specified. Although each agent $i$ only has its local information ($\tilde{F}_i,~\xi_i$), it can interact with its direct neighbors through a connected graph $\mathcal{G}=(\mathcal{N},\mathcal{E})$ to collaboratively solve problem \eqref{p1}, where $\mathcal{N}$ is the set of $N$ nodes and $\mathcal{E}$ is the set of edges.

Since problem \eqref{p1} is a distributed stochastic optimization on Riemannian submanifolds without convexity and linearity, the centralized Riemannian stochastic algorithms and the previous methods developed in the Euclidean space (such as \cite{shi2015proximal}) cannot be directly applied to solve it. In the following, we  will initiate a detailed literature review.

\subsection{Literature Review}
Since the seminal work of Robbins and Monro \cite{robbins1951stochastic} in the 1950s, stochastic approximation (SA) has become a fundamental methodology for solving stochastic optimization problems. The central idea of SA is to approximate exact gradients using randomly sampled gradient estimates. Owing to its ability to significantly reduce computational complexity and handle uncertainty, SA has found widespread applications in large-scale optimization problems, particularly in expected risk minimization for supervised learning \cite{bottou2010}. The dynamical analysis of generalized Robbins-Monro algorithms was recently extended to Riemannian settings in \cite{karimi2022dynamics}, providing theoretical foundations closely related to the present work.

In recent years, substantial attention has been devoted to centralized Riemannian stochastic optimization. For instance, \cite{6487381} introduced Riemannian stochastic gradient descent (RSGD) and established its almost sure convergence to critical points. Building upon RSGD, several variance-reduced methods have been developed to accelerate convergence, including \cite{NIPS2016_98e6f172, doi:10.1137/17M1116787, pmlr-v75-tripuraneni18a, pmlr-v80-kasai18a}. For optimization problems involving nonsmooth regularizers, \cite{JMLR:v23:21-0314} developed the R-ProxSGD and R-ProxSPB algorithms on the Stiefel manifold. For weakly-convex objectives, \cite{doi:10.1137/20M1321000} proposed a Riemannian stochastic subgradient method on the Stiefel manifold and further extended the convergence analysis to compact embedded submanifolds. More generally, for optimization over proximally smooth sets, including compact $C^2$ submanifolds, \cite{davis2020} developed stochastic proximal methods and analyzed their convergence properties from an Euclidean perspective. In Euclidean spaces, the transition from centralized to distributed optimization is typically enabled by consensus mechanisms based on weighted averaging; (see, eg, \cite{nemirovski2009robust, srivastava2011distributed, pu2021distributed, sayed2014adaptation, lei2022distributed, alghunaim2019distributed}). However, extending centralized Riemannian optimization methods to distributed settings is considerably more challenging. The nonlinear and nonconvex nature of Riemannian manifolds precludes the direct use of Euclidean averaging operations, therefore, weighted-average consensus protocols could be invalid. Although the Riemannian meta-expert framework proposed in \cite{hu2023minimizing} provides a general aggregation mechanism based on the Fréchet mean, it is tailored to centralized scenarios and does not account for the decentralized information structure and communication constraints inherent to distributed optimization. These challenges make the design of distributed algorithms for solving problem \eqref{p1} nontrivial.

The literature on distributed Riemannian optimization can be broadly categorized into intrinsic and extrinsic methods according to their analytical frameworks. Intrinsic methods are formulated entirely on the manifold itself, relying on concepts such as Fr\'echet mean, connections, and geodesics without referring to any embedding ambient space. Existing intrinsic approaches commonly require geodesic convexity (g-convexity) of the objective functions (e.g., \cite{shah2017distributed, chen2024online, 6332485, 11107888}), or alternatively impose the Polyak--Łojasiewicz (PL) condition when g-convexity is unavailable (e.g., \cite{wang2025distributed}). In contrast, extrinsic methods exploit the embedding of the manifold and can leverage tools from Euclidean optimization, leading to simpler computations and more efficient implementations. The following discussions focus on the extrinsic methods, to which our work belongs.

For a specific manifold like the Stiefel manifold, \cite{chen2023local} proposed a consensus protocol and proved that the iterates converge linearly to the induced arithmetic mean in a local region if the initial points are in that region. Based on this, a series of distributed Riemannian optimization algorithms on the Stiefel manifold have emerged, such as \cite{pmlr-v139-chen21g, chen2024decentralized, wang2022decentralized, wang2023} for the deterministic setting and \cite{pmlr-v139-chen21g, wang2022variance, hu2023decentralized, zhao2024stochastic} for the stochastic setting. 
In particular, \cite{pmlr-v139-chen21g} introduced distributed Riemannian stochastic gradient descent and distributed Riemannian gradient tracking, establishing convergence rates of $\mathcal{O}(1/\sqrt{k})$ and $\mathcal{O}(1/k)$, respectively.
Subsequently, \cite{wang2023} developed a distributed Riemannian subgradient method for nonsmooth weakly-convex objectives. Furthermore, variance-reduced stochastic algorithms proposed in \cite{wang2022variance, zhao2024stochastic} improved the convergence rate to $\mathcal{O}(1/k)$. The analyses of these methods rely heavily on structural properties specific to the Stiefel manifold, such as the constant Frobenius norm of feasible points and the availability of closed-form orthogonal projections, which no longer holds on general embedded submanifolds.

For embedded submanifolds, most of the existing distributed optimization methods are projection-based algorithms by treating the manifold as a proximally smooth set and employing projection operators to maintain feasibility. For smooth objectives, \cite{deng2023decentralized} proposed decentralized projected Riemannian gradient descent (DPRGD) and gradient tracking (GT) algorithms, establishing convergence rates matching those obtained on the Stiefel manifold. In addition, \cite{deng2023Douglas-Rachford} developed decentralized Douglas--Rachford splitting methods and achieved an $\mathcal{O}(1/k)$ convergence rate. Extending this framework to composite optimization, \cite{wang2024proxtrack} studied problems involving convex but nonsmooth regularizers and proposed a DRproxGT method. Beyond deterministic settings, \cite{deng2024stochastic} introduced a projection-based stochastic momentum algorithm for smooth optimization and established a convergence rate of $\mathcal{O}(1/k^{2/3})$. However, manifold projections generally do not admit closed-form expressions and must be computed by solving auxiliary optimization problems, thereby increasing computational overhead. Moreover, due to the nonconvex nature of embedded manifolds, the projection mapping is generally not globally unique. Besides, all of the aforementioned methods have not explicitly considered the impact of curvature bounds on algorithm performance.

For clarity, we summarize some representative distributed Riemannian algorithms in Table \ref{tab1}.

\begin{table*}
	\setlength{\abovecaptionskip}{0cm}
	\caption{Comparison of the state-of-the-art distributed Riemannian algorithms on the compact submanifolds for different problem settings.}
	\centering
	\small
	\begin{tabular}{p{2cm} |m{2.2cm}| m{2cm}| m{2cm}| m{2.5cm} |m{2.5cm}}
		\toprule [1pt]     
		Problem &literature &  Cost &Constraint &Alg. framework & Convergence  \\
		\hline
		\multirow{5}{*}{Deterministic}&\multirow{1}{*}{DRGTA\cite{pmlr-v139-chen21g}}&$L$-smooth&\multirow{2}{*}{Stiefel}&\multirow{2}{*}{Retraction-based}&$\mathcal{O}(\frac{1}{\sqrt{k}})$\\
		\cline{2-3}\cline{6-6}
		~&\multirow{1}{*}{DRSM\cite{wang2023}}&weakly-convex&~&~&$\mathcal{O}(\frac{1}{\sqrt{k}})$\\
		\cline{2-6} 
		~&\multirow{1}{*}{DPRGD/GT\cite{deng2023decentralized}}  &\multirow{2}{*}{$L$-smooth}&\multirow{3}{*}{\makecell{embedded  \\submanifolds}}&\multirow{3}{*}{Projection-based}&$\mathcal{O}(\frac{1}{\sqrt{k}})$/$\mathcal{O}(\frac{1}{{k}})$ \\
		\cline{2-2} \cline{6-6}
		~&\multirow{1}{*}{\cite{deng2023Douglas-Rachford}} &~&~&~&$\mathcal{O}(\frac{1}{{k}})$  \\
		\cline{2-3}\cline{6-6}
		~&\multirow{1}{*}{DRproxGT\cite{wang2024proxtrack}} &composite $f_i(x)+r(x)$ &~&~&$\mathcal{O}(\frac{1}{{k}})$ \\
		\midrule[1pt]
		\multirow{4}{*}{Stochastic}&\multirow{1}{*}{DRSGD\cite{pmlr-v139-chen21g}}&\multirow{2}{*}{$L$-smooth}&\multirow{2}{*}{Stiefel}&\multirow{2}{*}{Retraction-based}&$\mathcal{O}(1/\sqrt{k})$\\
		\cline{2-2} \cline{6-6}
		~&\multirow{1}{*}{DRSGT\cite{zhao2024stochastic}}&~&~&~&$\mathcal{O}(1/k)$(sample size$=[q^{-k}]$)\\
		\cline{2-6}
		~&DPRSRM \cite{deng2024stochastic}&$L$-smooth&\multirow{2}{*}{\makecell{embedded\\submanifolds}} &Projection-based&$\mathcal{O}(1/k^{2/3})$\\
		\cline{2-3} \cline{5-6}
		~&our work&  weakly-convex&~&Retraction-based&$\mathcal{O}((1+\kappa_g)/\sqrt{k})$\\
		\bottomrule [1pt]     
	\end{tabular}
	\label{tab1}
\end{table*}
\subsection{Main Contributions}
This paper develops retraction-based distributed Riemannian stochastic algorithms for weakly convex optimization over general Riemannian submanifolds. The main contributions are summarized as follows.

First, we propose a unified distributed Riemannian stochastic proximal framework for solving \eqref{p1} that gives rise to stochastic subgradient, proximal-point, and proximal linear methods. The framework employs retractions to maintain feasibility on the manifold, encompassing projection-like retractions \cite{absil2012projection} used in existing projection-based methods while potentially reducing computational cost.

Second, under suitable initialization and diminishing step sizes, the iterates remain within a prescribed local region and converge in expectation to a stationary point. Moreover, a convergence rate of $\mathcal{O}(\frac{1+\kappa_g}{\sqrt{k}})$ is established, explicitly characterizing the effect of the geodesic curvature bound $\kappa_g$ on algorithm performance.

Finally, we apply the proposed framework to distributed orthogonal sparse dictionary learning and generalized eigenvalue computation on generalized Stiefel manifolds. Numerical experiments demonstrate the effectiveness of the proposed algorithms and corroborate the theoretical results

\subsection{Paper Organization}
This paper is organized as follows. Section \ref{sec-2} introduces preliminaries of Riemannian geometry and the stationary measurement for weakly-convex functions. We reformulate the distributed Riemannian optimization problem and propose the algorithm framework with three specific algorithms in Section \ref{sec-3}. The main convergence results are given in Section \ref{sec-4}. Numerical experiments are provided in Section \ref{sec-5}, while some concluding remarks are given in Section \ref{sec-6}.

\section{Preliminaries}\label{sec-2}

{\noindent\bf Notations} Let $X^\top$ denote the transposition of $X$ and $tr(\cdot)$ denote the trace operator. Define the $N$-fold Cartesian product of $\mathcal{M}$ as $\mathcal{M}^N:=\mathcal{M}\times \cdots \times \mathcal{M}$. Let $I_n$ denote the $n\times n$ identity matrix and $\mathbf{1}_N$ denote the $N$-dimensional vector with all ones, respectively. $\mathcal{P}$ represents the orthogonal projection operator.   The Euclidean distance from a point $x \in \mathbb{R}^n$ to $\mathcal{M}$ is denoted by $d(x,\mathcal{M}):=\inf_{X\in \mathcal{M}} \|X-x\|$, where $\|\cdot\|$ is the Euclidean norm ($2$-norm for vectors and F-norm for matrices). For a function $\phi$ on $\mathcal{M}$, denote by $\nabla \phi$ the Euclidean gradient and $\text{grad}~\phi$ the Riemannian gradient.  

\subsection{Riemannian Geometry}
Before presenting the algorithm, we first introduce some basic concepts and geometric properties of the embedded submanifolds in this section. The reader can refer to literature \cite{AbsilMahonySepulchre+2008, lee2018introduction} for more information.

A topological space is called a manifold if each of the point has a neighborhood that is diffeomorphic to $\mathbb{R}^n$. For two manifolds such that $\mathcal{M} \subset \tilde {\mathcal{M}}$, if the manifold topology of $\mathcal{M}$ coincides with its subspace topology induced from the topological space $\tilde{\mathcal{M}}$, then $\mathcal{M}$ is called an embedded submanifold of $\tilde{\mathcal{M}}$. In this work, we consider the case that  $\mathcal{M}$ is an embedded submanifold  of $\mathbb{R}^n$. For any $X\in \mathcal{M}$, we have $i(X) \in \mathbb{R}^n$, where $i:\mathcal{M} \hookrightarrow \mathbb{R}^n$ is the natural inclusion of $\mathcal{M}$ in $\mathbb{R}^n$ (the inclusion $i$ is omitted in the following to simplify the notation). 

At any point $X\in \mathcal{M}$, tangent vectors are defined as the tangents of parametrized curves passing through $X$. Thus, a curve $\gamma,~\gamma(0) = X$ induces a tangent vector $\dot \gamma (0)$ at $X$. We denote the tangent space at a point $X$ by $T_X\mathcal{M}$, including all tangent vectors at $X$. For the case that $\mathcal{M}$ are submanifolds of $\mathbb{R}^n$, it is free to define $\gamma'(0):=\lim_{t\rightarrow 0} \frac{\gamma(t) - \gamma(0)}{t}$. The tangent space of $\mathcal{M}$ can be identified as $T_X\mathcal{M}\triangleq \{\gamma'(0):~\gamma \text{ is curve in }\mathcal{M},~\gamma(0) = x\}$ and its orthogonal complement to the Euclidean space is denoted by the normal space $N_X\mathcal{M}$.  The metric on the tangent space $T_X\mathcal{M}$ is induced from the Euclidean inner product $\langle u,v \rangle $, for $u,v \in T_X\mathcal{M}$. Then the length of a vector $v \in T_X\mathcal{M}$ is also equivalent to the Euclidean norm $\|v\|$.

The covariant derivative is a suitable generalization of the classical directional derivative. We can denote the covariant derivatives both on $\mathcal{M}$ and in the ambient space $\mathbb{R}^n$. To compare the two covariant derivatives along curves, we introduce the Gauss formula along a curve as follows.
\begin{lemma}\label{gf} 
	\cite[Corollary 8.3]{lee2018introduction} Suppose $\mathcal{M}$ is an embedded submanifold of $\mathbb{R}^n$, and $\gamma:[0,1]\mapsto \mathcal{M} \subset \mathbb{R}^n$. Let $\eta$ denote a smooth vector field along $\gamma$ that is everywhere tangent to $\mathcal{M}$. Then, the Levi-Civita connection $\nabla_{\dot\gamma}{\eta} $ on $\mathcal{M}$ is related to the classical vector derivative $\bar \nabla_{\dot \gamma} \eta$ on $\mathbb{R}^n$ as 
	\begin{align*}
	\nabla_{\dot\gamma}\eta=\bar \nabla_{\dot \gamma} \eta - \Pi(\dot\gamma, \eta),
	\end{align*}
	where $\Pi(\cdot,\cdot)$ is the second fundamental form on $\mathcal{M}$ and the Euclidean vector derivative $\bar \nabla_{\dot \gamma} \eta$ is given by
	$\bar \nabla_{\dot \gamma} \eta(X)= \lim_{t \to 0} \frac{\eta(X+t \dot \gamma) - \eta(X)}{t}.$
\end{lemma}

Then, we formally define the extrinsic geodesic curvature.
\begin{definition}\label{g-c}
	\cite[Chp. 8]{lee2018introduction} Suppose $\gamma$ is a unit-speed geodesic in $\mathcal{M}$. The extrinsic geodesic curvature of $\gamma$ is the length of the Euclidean acceleration vector field in the ambient space $\mathbb{R}^n$, which is the function $\kappa: I \mapsto \mathbb{R}$ given by $\kappa(t) = \|\bar{\nabla} _{\dot \gamma}\dot \gamma (t)\|$.
\end{definition}

We introduce the concept of exponential map to connect the tangent space and the manifold. The length of a curve $\gamma$ is defined as $L(\gamma):= \int_{1}^{0}  \sqrt{ \langle\dot{\gamma}(0), \dot{\gamma}(0)\rangle} dt$. Geodesics, the natural generalization of straight lines on a Riemannian manifold, are curves that locally minimize length between their endpoints. For a unique geodesic such that $\gamma(0) = X,~ \dot{\gamma}(0) = v \in T_X \mathcal{M}$, an exponential map $\text{Exp}_X$ maps a tangent vector $v$ to a point $Y: =  \text{Exp}_X(v) \in \mathcal{M}$. The Riemannian distance between $X,~Y$ is defined as $d_\mathcal{M}(X,Y):=\|\text{Exp}_X^{-1}(Y)\|$. 

Since the exponential map generally has no closed form, the retraction $\mathcal{R}_X: T_X \mathcal{M}\mapsto \mathcal{M}$ is introduced as a first-order approximation of the exponential map with the requirements that 1) $\mathcal{R}_X(0_X) = X$, where $0_X$ is the zero vector of $T_X\mathcal{M}$; 2) the differential of $\mathcal{R}_X$ at $0_X$ is the identity map, i.e., $\mathrm{D} \mathcal{R}_X (0_X) = id_{T_X \mathcal{M}} $, where the differential of the retraction is the linear operator $\mathrm{D}R_x(v):T_x\mathcal M\to T_{R_x(v)}\mathcal M \cong T_x\mathcal M$ that maps $v$ to $\mathrm{D}R_X(v)$ at point $X \in \mathcal{M}$. The following lemma collects useful properties of retraction maps.
\begin{lemma}\label{lem-r}
	\cite{BoumalNicolas} Let $\mathcal{R}_X: T_{X}\mathcal{M} \mapsto \mathcal{M}$ denote a retraction map at $X$. Then there exist constants $M_1, M_2 \geq 0$ such that for any $ X \in \mathcal{M}$ and $v \in T_{X}\mathcal{M}$,
	\begin{align}
	\|\mathcal{R}_{X}(v) -X\| \leq M_1\|v\|, \label{R-1}\\
	\|\mathcal{R}_X (v) -(X+v)\| \leq M_2\|v\|^2 \label{R-2}.
	\end{align}
\end{lemma}
The subsequent analysis relies only on the local properties of retractions in local neighborhood stated in Lemma~\ref{lem-r}. Hence, all results remain valid for locally defined retractions, provided that the iterates stay within their domains of definition.

We next define the average for manifold-valued variables. For $X_1,\dots,X_N \in \mathcal{M}$, denote the Euclidean average by $\bar X:=\frac{1}{N}\sum_{i=1}^N X_i$ and $\boldsymbol{\bar X}:=(\bm{1}_N\otimes I_n)\bar X$. Because $\mathcal{M}$ might be nonconvex, $\bar X$ need not lie on the manifold. We therefore employ the induced arithmetic mean (IAM) \cite{doi:10.1137/060673400}, defined as
\begin{equation}\label{iam}
\hat X \in \mathcal{X}_{IAM}:=\arg \min_{Y \in \mathcal{M}} \sum\nolimits_{i=1}^N \|Y - X_i\|^2 .
\end{equation}



\subsection{Stationarity Measurements for Weakly-convex Functions}

First, we give the definition of weakly-convex functions. 
\begin{definition}
	A function $\phi: \mathcal{M} \mapsto \mathbb{R}$ is called a $\rho$-weakly-convex function if the assignment $X\mapsto \phi(X)+\frac{\rho}{2}\|X\|^2$ is a convex function in the ambient Euclidean space.
\end{definition}

If $\phi$ is differentiable on $\mathcal{M}$, then the only tangent vector that satisfies $\langle \text{grad}\  \phi(X), v \rangle = D \phi(X)[v]$ ($\forall v \in T_X \mathcal{M}$) is called a Riemannian gradient $\text{grad}\  \phi(X)$, where $D \phi(X)[v] = \frac{d \phi(\gamma(t))}{dt} \big|_{t=0}$ such that $\dot \gamma(0) = v$. For embedded submanifolds, we have an equivalent form of the Riemannian gradient \cite{AbsilMahonySepulchre+2008}: \begin{equation}\label{rg}
\text{grad}\ \phi(X) = \mathcal{P}_{T_X \mathcal{M}}(\nabla \phi(X)),
\end{equation}where $\mathcal{P}_{T_X\mathcal{M}}$ denotes the orthogonal projection onto the tangent space $T_X\mathcal{M}$.

When the function is nonsmooth, we introduce the concept of Riemannian subdifferential and subgradient. 

If $\phi: \mathcal{M} \mapsto \mathbb{R}$ is Lipschitz continuous in the ambient space, we adopt the notion of Riemannian generalized subdifferential from \cite{doi:10.1137/16M1108145}. For any $X \in \mathcal{M}$, let $\hat{\phi}_X:=\phi \circ \mathcal{R}_X$. Since $\hat{\phi}_X$ is Lipschitz continuous on the Hilbert space $T_X\mathcal{M}$, the Clarke generalized directional derivative of $\phi$ at $X$ along $v \in T_X\mathcal{M}$ is given by
$\phi^\circ(X;v)
:=\hat{\phi}_X^\circ(0_X;v)
=\limsup{\substack{\zeta\to 0_X\ t\downarrow 0}}
\frac{\hat{\phi}_X(\zeta+tv)-\hat{\phi}_X(0_X)}{t}.$
The Riemannian subdifferential $\partial_R \phi(X)$ is defined as $\partial_R \phi(X) = \partial \hat{\phi}_X(0_X)$, where $\partial \hat{\phi}_X(0_X):=\{g \in T_X\mathcal{M}| \langle g, v \rangle \leq \hat{\phi}_X^\circ (0_X;v)~\text{for all}~v \in T_X \mathcal{M}\}$. Any element of $\partial_R \phi(X)$ is called a Riemannian subgradient of $\phi$ at X. 

If $\phi$ is weakly-convex, then $\psi(X) = \phi(X)+\frac{\rho}{2}\|X\|^2$ is convex, and the Euclidean subdifferential of $\phi$ can be expressed as $\partial \phi(X) =\partial \psi(X)- \rho X$. Combining the characterization of weakly-convex functions in \cite[Prop.~4.5]{doi:10.1287/moor.8.2.231} with \cite[Th.~5.1]{whYang}, the Riemannian subdifferential of $\phi$ admits the representation
\begin{equation}\label{subd}
\partial_R \phi(X) = \mathcal{P}_{T_X\mathcal{M}}[\partial \phi(X)],~\forall X \in \mathcal{M}.
\end{equation}
By \cite{doi:10.1137/20M1321000}, $X\in \mathcal{M}$ is called a stationary point of $\phi$ if $0 \in \partial_R \phi(X).$

Next, we derive the stationary measurement of problem \eqref{p1}.  Let $\lambda >0$, we define the Moreau envelope and proximal map for problem \eqref{p1}:
\begin{subequations}\label{Me}
	\begin{align}
	f_\lambda (X):= \min_{Y\in \mathcal{M}}\{f(Y)+\frac{1}{2\lambda}\|Y-X\|^2\},\\
	P_{\lambda f}(X):=\arg \min_{Y \in \mathcal{M}}\{f(Y)+\frac{1}{2\lambda}\|Y - X\|^2\}.
	\end{align}
\end{subequations}
By the first-order optimality condition of \eqref{Me}, we  have $\text{dist}(0, \partial_R f(P_{\lambda f}(X)))\leq \frac{1}{\lambda}\|P_{\lambda f}(X) - X\|$ \cite{wang2023}.
Thus, when $\|P_{\lambda f}(X) - X\|$ is sufficiently small, $X$ is near some point $P_{\lambda f}(X) \in \mathcal{M}$ that is nearly stationary for $f$.
\begin{definition}
	$\bold{X}=[X_1^\top,\dots,X_N^\top]^\top \in \mathcal{M}^N$ is called an $\epsilon$-nearly stationary point of problem \eqref{p1} if it satisfies $\frac{1}{\lambda^2}\|P_{\lambda f}(\hat{X}) - \hat{X}\|^2 \leq \epsilon,~~\frac{1}{N}\|\bold{X} - \bold{\hat X}\|^2 \leq \epsilon.$
\end{definition}

\section{Distributed Riemannian Stochastic Proximal Algorithm}\label{sec-3}
In this section, we reformulate problem \eqref{p1} and develop a distributed Riemannian stochastic proximal framework. The distributed stochastic subgradient, proximal-point, and proximal linear methods are then derived as special instances of the proposed framework.
\subsection{Problem Reformulation}
Let $X_i \in \mathcal{M}$ be a local copy of the optimization variable of each agent $i$. Each agent can only have access to its local information but interacts with the direct neighbors through a connected graph $\mathcal{G}=(\mathcal{N},\mathcal{E})$, where $\mathcal{N} = \{1,\dots,N\}$ and $\mathcal{E} \subset \mathcal{N} \times \mathcal{N}$. We reformulate the optimization problem \eqref{p1} as 
\begin{align}\label{p-2}
\min_{X_i \in \mathcal{M}}&\ \frac{1}{N}\sum_{i=1}^Nf_i(X_i),~s.t.~ X_i=X_j,~\forall (i,j) \in \mathcal{E},
\end{align}

which is equivalent to problem \eqref{p1} \cite{wang2022decentralized}. The local cost function of agent $i$ is given by  $f_i(X):=\mathbb{E}_{\xi_i} \tilde{F}_i(X,\xi_i)$, where $\xi_i$ is a random variable on $(\Omega_i, \mathcal{F}_i, \mathbb{P})$. The dependence on $i$ allows different agents to possess different local data distributions.
Given an underlying communication network $\mathcal{G}=(\mathcal{N},\mathcal{E})$, let $W:=\{w_{ij}\}_{N\times N}$ be the associated  weight matrix satisfying $w_{ij}>0$ if $(i,j)\in \mathcal{E}$ or $i=j$, and $w_{ij}=0$ otherwise.

Subsequently, we make some assumptions on the cost functions which are standard in the literature \cite{pmlr-v139-chen21g, wang2023}, while the assumption of smoothness is relaxed to weakly-convex. 
\begin{assumption}
	\label{assu-3.1}
	The local cost functions $f_i$ are $L$-Lipschitz continuous and $\rho$-weakly-convex on the ambient space of $\mathcal{M}$. For any $i\in \mathcal{N}$, $f_i$ is uniformly bounded from below.
\end{assumption}
As $f$ is denoted by the average of all $f_i$, $f$ is also $L$-Lipschitz continuous and $\rho$-weakly-convex.

The following are assumptions on the communication graph $\mathcal{G}$, which are standard in \cite{pmlr-v139-chen21g, wang2022variance, zhao2024stochastic}.
\begin{assumption}\label{assu-3.2}
	Suppose the graph $\mathcal{G}$ is undirected and connected, and $W$ is a nonnegative matrix satisfying $W=W^\top$ and $W \bm{1}_N = W^\top \bm{1}_N = \bm{1}_N$.
\end{assumption}
\begin{remark}\label{re-W}
	According to Assumption \ref{assu-3.2}, any power of the matrix $W$, i.e., $W^t,~t \geq 1$, is also symmetric and doubly stochastic. Since the graph is connected, all the eigenvalues of $W$ are in $(-1,1]$, and the second largest singular value $\sigma_2\in [0,1)$ \cite{1406483}, which is also the spectral norm of $W - \frac{1}{N}\bm{1_N} \bm{1_N}^\top$ \cite[Th. 5.1]{doi:10.1137/060678324}.
\end{remark}

\subsection{The Algorithm Framework}
To solve the problem, we adopt the stochastic one-sided model framework of \cite{doi:10.1137/18M1178244, davis2020} in the Riemannian setting. The resulting stochastic models provide local approximations of the expected-value functions $f_i$ at the current iterate; for example, stochastic (sub)gradient methods are induced by linear approximation models.

For each $i \in \mathcal{N}$, fix a probability space $(\Omega_i, \mathcal{F}_i, \mathbb{P})$, and equip $\mathbb{R}^n$ with the Borel $\sigma$-algebra. The model $F_{i,X}(\cdot ,\xi_i)$ is a measurable function $(X, Y, \xi_i) \mapsto F_{i,X}(Y, \xi_i)\in \mathbb{R}$ on $\mathbb{R}^n \times \mathbb{R}^n \times \Omega_i $ and satisfies the following conditions.
\begin{definition}\label{s-model}
	For real constants $\tau,~\rho,~L>0$,  $F_{i,X}(\cdot ,\xi_i)$ satisfies the following conditions.
	\begin{itemize}
		\item[a.] The generated realizations $\xi_i^{[1]}, \xi_i^{[2]}, \dots \sim \mathbb{P}$ are i.i.d. 
		\item[b.] $F_{i,X}(\cdot ,\xi_i)$ satisfies one-sided accuracy, i.e., \begin{align*}
		\mathbb{E}_{\xi_i}[F_{i,X}(X, \xi_i)] &= f_i(X)\\
		\mathbb{E}_{\xi_i}[F_{i,X}(Y, \xi_i)-f_i(Y)] &\leq \frac{\tau}{2}\|Y - X\|^2.
		\end{align*}                    
		\item[c.] For any fixed $X\in \mathcal{M}$, the function $F_{i,X}(\cdot, \xi_i)$ is $\rho$-weakly convex for almost every $\xi_i \sim \mathbb{P}$.
		\item[d.] For any fixed $X\in \mathcal{M}$, $F_{i,X}(\cdot, \xi_i)$ is $L$-Lipschitz continuous for almost every $\xi_i \sim \mathbb{P}$, i.e., $|F_{i,X}(X, \xi_i) - F_{i,X}(Y, \xi_i)|\leq L\|X-Y\|$.
	\end{itemize}
\end{definition}

Let $X_{i,k}$ denote the iterate maintained by agent $i$ at iteration $k$. Given $X_{i,k}$, the search direction $v_{i,k}\in T_{X_{i,k}}\mathcal M$ is obtained by solving
\begin{equation}\label{prox}
v_{i,k} = \arg \min_{v\in T_{X_{i,k}}\mathcal{M}}\{F_{i,X_{i,k}}(X_{i,k}+v, \xi_i) +\frac{\beta_k}{2}\|v\|^2\},
\end{equation} where the parameter $\beta_k$ acts as a searching step size. Since $\mathcal{M}$ is an embedded manifold in $\mathbb{R}^n$, the addition $X_{i,k}+v$ is valid.  In particular, when $\beta_k > \rho$, Definition \ref{s-model} (c) implies that the objective function of \eqref{prox} is $(\beta_k-\rho)$-strongly convex in $v$. Hence, subproblem \eqref{prox} admits a unique optimality.

To enforce consensus on $\mathcal{M}$, we consider the consensus potential $\min_{\mathbf{X} \in \mathcal{M}^N} h_t(\bold{X}):=\frac{1}{2}\sum_{i=1}^Nh_{i,t}(\bold{X})$, where $h_{i,t}(\bold{X}):=\frac{1}{2} \sum_{j=1}^N W_{ij}^t\|X_i-X_j\|^2$ measures the local disagreement of agent $i$. Here, $W_{ij}^t$ represents the ($i,j$)-th element of $W^t$, the $t$th-power of the weight matrix.

Combining the direction in \eqref{prox} with a consensus dynamic gives the update of $X_{i,k}$ as
\begin{equation}\label{e-iter}
X_{i,k+1} = \mathcal{R}_{X_{i,k}}\left(-\alpha \text{grad}~h_{i,t}(\bold{X}_k)+v_{i,k} \right).
\end{equation}

We summarize the procedure in Algorithm \ref{alg:1}.
\begin{algorithm}
	\caption{Distributed Riemannian Stochastic Proximal Gradient Algorithm}
	\label{alg:1}
	\begin{algorithmic}
		\REQUIRE For all $i\in \mathcal{N}$, set $\bold{X}_{0} \in \mathcal{S}$, where the local region $\mathcal{S}$ will be formally defined in \eqref{l_r}. Let the step sizes $\alpha, \beta_k>0$, and $t \geq \left[\log_{\sigma_2}\left(\frac{1}{5 \sqrt{N}}\right)\right]$.
		\FOR {$k=0,1,2,\dots,K$}
		\STATE Step $1$: Proximal step to get the search direction: \eqref{prox};
		\STATE Step $2$: Update the optimal variable $X_{i,k+1}$ by \eqref{e-iter}. 
		\ENDFOR
	\end{algorithmic}
\end{algorithm}

\begin{remark}\label{rem-alg}
	(1). Although stochastic one-sided models have also been studied in Euclidean spaces \cite{doi:10.1137/18M1178244, davis2020}, their extension to Riemannian manifolds is not straightforward. The key distinction is that the optimization subproblem is formulated on the tangent space and produces a search direction $v_{i,k}$ rather than a proximal point in the ambient space, necessitating an additional retraction step to ensure manifold feasibility.
	
	\noindent (2). 
	We next compare Algorithm~\ref{alg:1} with existing distributed Riemannian methods. 
	
	Most algorithms for embedded submanifolds \cite{deng2023decentralized, wang2024proxtrack} employ the projection
	$\mathcal{P}_\mathcal{M}(X) = \arg \min_{Y\in \mathcal{M}}\|X-Y\|$ to enable feasibility. In contrast, Algorithm~\ref{alg:1} is based on manifold retractions. Although orthogonal projections onto $\mathcal{M}$ induce a class of projection-like retractions \cite{absil2012projection}, the proposed framework accommodates a broader family of retractions, offering greater flexibility and potentially lower computational cost. 
	We show the computational cost of projection and retraction on several representative manifolds in Table~\ref{tab:proj_vs_retr_complexity}. This computational gap becomes particularly significant when $n \gg p$, or in large-scale problems where projection steps are performed repeatedly.
	Moreover, projection-based approaches may be computationally demanding since projections generally do not admit closed-form expressions and are not globally unique owing to the nonconvexity of the manifold. 
	
	When $\mathcal{M}$ is the Stiefel manifold and the distribution of each $\xi_i$ is known, the distributed Riemannian subgradient method in \cite{wang2023} can be recovered from Algorithm~\ref{alg:1} by adopting a linear approximation model. Therefore, Algorithm~\ref{alg:1} applies to a broader class of optimization problems while providing a computationally tractable retraction-based framework.
	
	\begin{table}[ht]
		\centering
		\caption{The computational complexity of both projection and retraction \& $\kappa_g$ of representative manifolds.}
		\begin{tabularx}{\columnwidth}{lcXcr}
			\toprule
			Manifold & Operation & Method & Complexity & $\kappa_g$\\
			\midrule
			
			Sphere 
			& Proj. 
			& Normalization $x / \|x\|$ 
			& $\mathcal{O}(n)$ 
			&\multirow{2}{*}{$\frac{1}{r}$}
			\\$\mathbb{S}^{n-1}(r)$ 
			& Retr. 
			& First-order update + normalization 
			& $\mathcal{O}(n)$ 
			&~
			\\
			\midrule
			
			Oblique 
			& Proj. 
			& Column-wise normalization 
			& $\mathcal{O}(np)$ 
			&\multirow{2}{*}{$1$}
			\\manifold 
			& Retr. 
			& Incremental normalization 
			& $\mathcal{O}(np)$ 
			&~
			\\
			\midrule
			
			Stiefel 
			& Proj. 
			& Polar / SVD 
			& $\mathcal{O}(np^2 + p^3)$ 
			&$1$
			\\$\mathrm{St}(n,p)$
			&Retr.
			& QR/Cayley 
			& $\mathcal{O}(np^2)$ 
			\\
			\midrule
			
			Grassmann 
			& Proj. 
			& SVD-based 
			& $\mathcal{O}(np^2 + p^3)$ 
			&\multirow{2}{*}{$1$}
			\\ $\mathrm{Gr}(n,p)$
			& Retr. 
			& QR-based
			& $\mathcal{O}(np^2)$ 
			&~
			\\
			
			\bottomrule
		\end{tabularx}
		\label{tab:proj_vs_retr_complexity}
	\end{table}

\end{remark}

\subsection{Algorithmic Examples}\label{a-ex}
We next derive three distributed algorithms from Algorithm~\ref{alg:1} and state the corresponding assumptions to ensure that Definition~\ref{s-model} is satisfied. Table~\ref{tab2} lists three one-side stochastic models, illustrating how different model choices lead to different algorithmic instances of the proposed framework.

\begin{table}[ht]
	\centering
	\setlength{\abovecaptionskip}{0cm}
	\caption{Examples for stochastic models, where $G_i(X,\xi_i)$ denote the stochastic Riemannian subgradient on tangent space, and $J_i(X, \xi_i) = \nabla c_i(X, \xi_i)$ denote the Jacobian of $c_i$.}
	\renewcommand{\arraystretch}{1.5}
	\label{tab2}
	\begin{tabularx}{\columnwidth}{l X r}
		\toprule[1pt]
		Algorithm & Cost function & Model $F_{i,X}(X+v, \xi_i)$\\ \midrule[1pt] 
		Subgradient& $ \mathbb{E}_{\xi_i}[\tilde{F}_i(X,\xi_i)]$ & $\tilde{F}_i(X, \xi_i)+\langle G_i(X, \xi_i), v\rangle$\\ \hline
		\vspace{1pt} 
		Proximal point& $\mathbb{E}_{\xi_i}[\tilde{F}_i(X,\xi_i)]$ & $\tilde{F}_{i}(X+v, \xi_i)$\\ \hline
		Prox-linear& $\mathbb{E}_{\xi_i}[H(c_i(X,\xi_i),  \xi_i)]$ & $H(c_i(X, \xi_i)+J_i(X,\xi_i)v,\xi_i)$\\
		\bottomrule[1pt]
	\end{tabularx}
\end{table}

{\bf Distributed Riemannian stochastic subgradient:} Let $Y = X+v,~v \in T_X \mathcal{M}$ and set the model as $F_{i,X}(Y , \xi_i) = \tilde{F}_i(X, \xi_i)+\langle G_i(X, \xi_i), Y-X\rangle$,  where $G_i(X,\xi_i)\in T_X \mathcal{M}$ denotes the approximate subgradient sampled by a stochastic oracle. We impose the following assumptions for each agent.
\begin{itemize}
	\item[A1)] The generated realizations $\xi_i^{[1]}, \xi_i^{[2]}, \dots \sim \mathbb{P}$ are i.i.d.;
	\item[A2)] For any $X\in \mathcal{M}$, $\mathbb{E}_{\xi_i}[G_i(X,\xi_i)] \in \partial_R f_i(X)$;
	\item[A3)] For any $X\in \mathcal{M}$, $f_i(X)$ is $\rho$-weakly-convex;
	\item[A4)] The sampled gradients are bounded, i.e., there exists a constant $L\ge 0$ s.t.  $\|G_i(X,\xi_i)\| \le L$.
\end{itemize}
Definition \ref{s-model}(a) and (d) are immediate from (A1) and (A4). Definition \ref{s-model}(c) holds since $F_{i,X}(Y , \xi_i)$ is convex associated to $Y$. By (A2) and (A3), there holds
\begin{align*}
\mathbb{E}_{\xi_i}[F_{i,X}(Y , \xi_i) - f_i(Y)] \leq& f_i(X) - f_i(Y) + \langle g_X, Y-X\rangle \\
\leq& \tfrac{\rho}{2}\|Y-X\|^2,~ g_X \in \partial_R f_i(X),
\end{align*} which implies that Definition \ref{s-model}(b) holds with $\tau = \rho$.  

{\bf Distributed Riemannian stochastic proximal point:} Set $F_{i,X}(Y , \xi_i) = \tilde{F}_i(Y, \xi_i)$, where $Y = X+v,~v \in T_X \mathcal{M}$. Assume that the following properties are satisfied.
\begin{itemize}
	\item[B1)] The generated realizations $\xi_i^{[1]}, \xi_i^{[2]}, \dots \sim \mathbb{P}$ are i.i.d.;
	\item[B2)] For any $Y \in \mathcal{M}$, there holds $\mathbb{E}_{\xi_i}[\tilde{F}_{i}(Y,\xi_i)] = f_i(Y)$;
	\item[B3)] Each function $\tilde{F}_{i}(\cdot , \xi_i)$  is $\rho$-weakly-convex;
	\item[B4)] Each function $\tilde{F}_{i}(\cdot, \xi_i)$ is $L$-Lipschitz continuous.
\end{itemize}
We check the definition of stochastic model, and clearly see that Definition \ref{s-model} (a)-(d) are immediately satisfied with $\tau=0$.

{\bf Distributed Riemannian stochastic proximal linear:} Consider the optimization problem \eqref{p1} with $f_i(X):=E_{\xi_i}[H_i(c_i(X,\xi_i),\xi_i)]$, where $c_i: \mathcal{M} \mapsto \mathbb{R}^m$ is $C^1$-smooth and $H_i:\mathbb{R}^m \mapsto \mathbb{R}$ is convex\footnote{This is a common type of weakly-convex functions \cite[Lem. 4.2]{drusvyatskiy2019}}. By applying the function $H_i$ to a first-order approximation of $c_i$, the proximal linear model is given as 
$F_{i,X}(Y, \xi_i) = H_i(c_i(X, \xi_i)+J_i(X,\xi_i)(Y-X),\xi_i),$ where $Y = X+v,~v \in T_X \mathcal{M}$. Suppose the following assumptions hold.
\begin{itemize}
	\item[C1)] The generated realizations $\xi_i^{[1]}, \xi_i^{[2]}, \dots \sim \mathbb{P}$ are i.i.d.;
	\item[C2)] Each function $H_i(\cdot, \xi_i)$ is $l_1$-Lipschitz;
	\item[C3)] Every $c_i(\cdot, \xi_i)$ are $l_2$-smooth;
	\item[C4)] There exists a constant $A\ge 0$ such that $\|\nabla c_i(X,\xi_i)\| \leq A$ for $\forall i$. 
\end{itemize}
Definition \ref{s-model}(a) and (c) hold trivially with the convexity of $H_i$. By (C2) and (C3), there holds
\begin{align*}
&\mathbb{E}_{\xi_i}[H_i(c_i(X,\xi_i)+J_i(X,\xi_i) (Y-X),\xi_i) - H_i(c_i(Y,\xi_i),\xi_i)] \\
&\leq l_1\mathbb{E}_{\xi_i}[c_i(X,\xi_i) -c_i(Y,\xi_i) + J_i(X,\xi_i) (Y-X)] \\
&\leq\frac{l_1l_2}{2}\|X-Y\|^2,
\end{align*} which implies Definition \ref{s-model}(b) with $\tau = l_1l_2$. Definition \ref{s-model}(d) is derived from (C4) with $L = l_1A$.

\section{Main Results}\label{sec-4}
In this section, we analyze the consensus and convergence behavior of the proposed algorithmic framework and establish its convergence rate. A key ingredient in the analysis is a normal vector inequality involving the geodesic curvature bound, stated in Lemma~\ref{nv}.

\subsection{Key Lemmas}
The convergence analysis of distributed optimization algorithms on compact embedded submanifolds faces two main challenges. First, existing analyses on the Stiefel manifold do not readily extend to general submanifolds, as they rely heavily on geometric properties specific to the Stiefel manifold, including its homogeneous structure \cite{doi:10.1137/060673400} and he Lipschitz-type inequality of the retraction \cite{chen2023local}. Moreover, explicit expressions for both projections and retractions are available on the Stiefel manifold but are generally unavailable on arbitrary embedded submanifolds.
Second, existing works on compact submanifolds \cite{deng2023decentralized, wang2024proxtrack} by treating the manifold as an $R$-proximally smooth subset of the ambient Euclidean space and exploiting properties of the projection operator. Since our algorithm framework is built upon retractions instead of metric projections, the analytical tools developed for projection-based methods are not directly applicable. Moreover, we further explicitly characterize the influence of the geodesic curvature bound on algorithmic convergence.

To overcome these challenges and analyze the impact of Riemannian geometry on the algorithm convergence, we exploit the second fundamental form to establish a normal vector inequality and a local projection-Lipschitz property on $\mathcal{M}$. 

Recall Lemma \ref{gf} given in Section \ref{sec-2}. When the curve $\gamma$ ($\gamma(0) = X$) is a unit-speed geodesic on $\mathcal{M}$, and $\eta: = \dot{\gamma}$ is its velocity vector field, by the Gauss formula we have
\begin{align}\label{gf-geo}
\bar{\nabla}_{\eta} \eta = \nabla_{\eta} \eta + \Pi(\eta, \eta) = \Pi(\eta, \eta).
\end{align}  
Thus, the second fundamental form $\Pi(\eta, \eta)$ represents the Euclidean acceleration of the geodesic $\gamma$ at $X$ \cite[Prop.~8.10]{lee2018introduction}, characterizing how $\gamma$ bends within the embedding space. 
	
	According to Definition \ref{g-c}, we refer to the length $\|\Pi(\eta, \eta)\|$ as the extrinsic geodesic curvature of $\gamma$ and assume it is uniformly bounded.

\begin{assumption}\label{k}
	The geodesic curvature along a geodesic $\gamma$ has an upper bound $\kappa_g\ge 0$, i.e., for any $X \in \mathcal{M}$, $\gamma(0) = X$ and $\eta = \dot \gamma$, it holds
		$\|\Pi(\eta, \eta)\| \leq \kappa_g \|\eta\|^2$, where $\|\eta\|: =  \sqrt{\langle \eta(t), \eta(t) \rangle}$ denotes the norm of vector field $\eta$ along a curve. 
\end{assumption}  
\begin{remark}\label{r-3}
	(1). The constant $\kappa_g$ can be interpreted as the operator norm of the second fundamental form. Since $\Pi$ is bilinear and continuous, $\kappa_g = \|\Pi\|_{op} := \sup_{X\in \mathcal{M}, \eta \in T_X \mathcal{M}}\max_{\|\eta\| \le 1} \|\Pi(\eta, \eta)\|$ is finite by the compactness of $\mathcal{M}$. Hence, Assumption~\ref{k} is mild and automatically satisfied for compact embedded submanifolds.
	
	\noindent(2). When $\mathcal{M}$ is an embedded hypersurface of $\mathbb{R}^n$, the second fundamental form can be computed using \cite[Prop.~8.23]{lee2018introduction}. 
	
	\noindent(3). For more general submanifolds, we can compute the second fundamental form by the Weingarten map $\mathfrak{W}_X(\cdot, \cdot)$ that maps $w \in N_X \mathcal{M}$ and $v \in T_X \mathcal{M}$ to a tangent vector and satisfies $\langle w, \Pi(v,u) \rangle = \langle \mathfrak{W}_X(w, v), u \rangle$, $u \in T_X \mathcal{M}$. 
		According to the Weingarten equation \cite[Prop. 8.4]{lee2018introduction},
		$\mathfrak{W}_X(w, v) = - \mathcal{P}_{T_X \mathcal{M}}(\bar \nabla_v \omega)$, 
		where $\omega$ is any local extension of $w$ to a normal vector field. 
		
		Therefore, for many commonly used manifolds, $\kappa_g$ can be explicitly estimated once the second fundamental form is available, and we give some examples in Table~\ref{tab:proj_vs_retr_complexity}.
	
\end{remark}

We next establish the key inequalities in Lemma~\ref{nv}. While these inequalities resemble \cite[Eq.~(2.2)]{deng2023decentralized}, which originates from \cite{00823375}, our derivation is based on the geometry of embedded Riemannian manifolds and explicitly incorporates the geodesic curvature through $\kappa_g$.
\begin{lemma}[Normal vector inequality]\label{nv}
	Suppose Assumption \ref{k} holds. For any $X,Y \in \mathcal{M}$ and a normal vector $w \in N_X \mathcal{M}$, there holds \begin{equation}\label{r_cur}
	\langle w, Y-X \rangle \leq  \tfrac{\kappa_g\|w\| }{2} d^2_\mathcal{M}(X,Y) ;
	\end{equation}
	Moreover, $D = \frac{2-\sqrt{2}}{\kappa_g}$ is a positive constant. For every fixed $X \in \mathcal{M}$ and $\Omega = \text{Exp}_X(B(0_X,D))$, where $B(0_X,D)$ denotes a ball with $0_X$ as its origin and $D$ as its radius, if we have $Y \in \Omega$, then there holds
	\begin{equation}\label{cur}
	\langle w, Y-X \rangle \leq \kappa_g \|w\|  \|Y-X\|^2.
	\end{equation}
\end{lemma}
\begin{proof}
	For $X,~Y \in \mathcal{M}$, there exists a  geodesic $\gamma:[0,1] \mapsto\mathcal{M}$ connecting $X$ and $Y$. We denote the tangent vector along $\gamma(t)$ as $\eta(t): = \dot \gamma(t)$. From the classical calculus in Euclidean spaces, we have
	$      Y-X = \int_0^1 \eta(t) dt$, which further gives
	\begin{align*}
	Y-X-\eta(0)  = &   \int_0^1 \left( \eta(t) - \eta(0) \right)dt= \int_0^1 \int_0^t \nabla_{\eta} \eta(s) ds dt,
	\end{align*}
	where the last equality follows from the fundamental theorem of calculus that $\eta(t) - \eta(0)  =  \int_0^t \nabla_{\eta} \eta(s) ds.$
	
	Since $\eta$ is a tangent field of geodesic $\gamma$, by \eqref{gf-geo} we have $Y-X-\eta(0) =  \int_0^1 \int_0^t \Pi(\eta(s),\eta(s)) ds dt.$ Taking norm of both sides implies
	\begin{align}\label{dist}
	&\| Y-X-\eta(0) \| =   \left\| \int_0^1\int_0^t \Pi(\eta(s),\eta(s)) ds dt \right \| \notag\\
	\overset{(a)}{\leq}& \int_0^1 \int_0^t \|\Pi(\eta(s),\eta(s))\| ds dt\overset{(b)}{\leq}  \kappa_g \int_0^1 \int_0^t \|\eta(s)\|^2 ds dt \notag\\
	=   &  \kappa_g \|\eta(0)\|^2 \int_0^1 \int_0^t ds dt = \frac{\kappa_g}{2} \|\eta(0)\|^2,
	\end{align}
	where (a) holds since $\|\eta(t)\|$ is a constant when $\gamma$ is a geodesic 
		\cite[Corl. 5.6]{lee2018introduction}, and (b) is derived by Assumption \ref{k}.
	
	Since $w\in N_X \mathcal{M}$ and $\eta(0)\in T_X \mathcal{M}$, together with \eqref{dist} we can derive that
		\begin{align}\label{eq13}
		\langle \tfrac{w}{\|w\|}, Y-X\rangle =   & \langle \tfrac{w}{\|w\|}, Y-X-\eta(0)\rangle 
		\le   \|Y-X-\eta(0)\|  \notag\\
		\le &  \tfrac{\kappa_g}{2}\|\eta(0)\|^2 =  \tfrac{\kappa_g}{2}d^2_{\mathcal M}(X,Y).
		\end{align}
	If we further have $Y \in \Omega$, \eqref{dist} yields
	\begin{align}\label{egeqr}
	\| Y-X\| \ge \|\eta(0)\| - \frac{\kappa_g}{2}\|\eta(0)\|^2  \ge \frac{1}{\sqrt{2}}\|\eta(0)\|,
	\end{align}
	where the last inequality is due to $\|\eta(0)\| = d_{\mathcal M}(X,Y) \le \frac{2-\sqrt{2}}{\kappa_g}$. 
	Plugging this into \eqref{eq13} implies \eqref{cur}.
\end{proof}

\begin{remark}\label{r-4}
	Lemma~\ref{nv} shows that, for two points $X,Y\in\mathcal{M}$, the displacement of $Y$ from $X$ in a normal direction can be bounded by their Riemannian (or Euclidean) distance, scaled by a curvature-dependent constant.
		It relates the local proximality of $\mathcal M$ to the geodesic curvature bound $\kappa_g$. While the uniform normal inequality is typically expressed in terms of the proximal smoothness radius $R$ \cite{deng2023decentralized, wang2024proxtrack}, which is a global constant and is often difficult to compute, $\kappa_g$ explicitly captures the local geometry of the embedding through the second fundamental form and can be quantified for many manifolds of practical interest (see Remark~\ref{r-3} and Table~\ref{tab:proj_vs_retr_complexity}). 
\end{remark}

For any $X\in \mathcal{M}$, $x := X+ d(x,\mathcal{M}) \frac{w}{\|w\|}$ denotes the point deviating $\mathcal{M}$ along the normal direction $w$, where $d(x,\mathcal{M})$ representing the deviation distance. If there satisfies $d(x,\mathcal{M}) < \frac{1}{2\kappa_g}$, then by \eqref{cur} we have $\langle \frac{w}{\|w\|}, Y-X\rangle \le \frac{1}{2d(x, \mathcal{M})}\|X-Y\|^2$, which further implies $\langle 2d(x, \mathcal{M}) \frac{w}{\|w\|}, Y-X\rangle \le \|X-Y\|^2$. By adding $d^2(x, \mathcal{M})$ on both sides and rearranging the terms, it holds that
	\[\|x-Y\| = \left\|X+d(x,\mathcal{M}) \tfrac{w}{\|w\|} - Y\right\| \geq d(x,\mathcal{M}),~\forall Y\in \Omega,\] 
	where $\Omega$ is defined in Lemma \ref{nv}. It indicates that the distance from $x$ to any point in $\Omega$ is greater than $d(x, \mathcal{M})$, that is, $X$ is the unique orthogonal projection of $x$ in $\Omega$. For simplicity, we denote by $\mathcal{P}|_{\Omega}(x)$ the projection of $x$ onto $\mathcal{M}$ restricted in $\Omega$. The geometry expression is shown in Fig.\ref{local projecion}.
\begin{figure}[h]
	\centering
	\setlength{\belowcaptionskip}{-0.2cm}
	\includegraphics[width=0.7\linewidth]{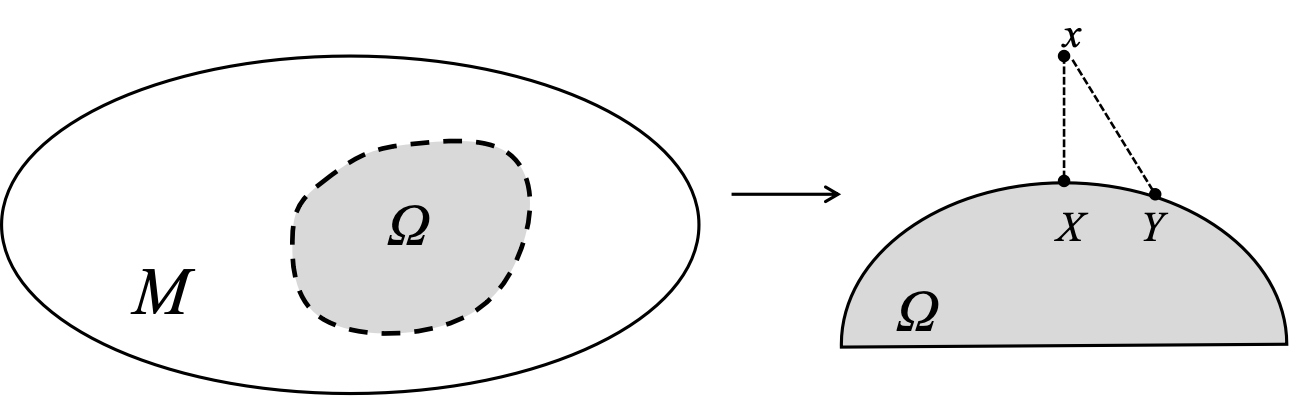}
	\caption{The geometry expression of Lemma \ref{nv}.}
	\label{local projecion}
\end{figure}

Under sufficient conditions, the Euclidean distance between two points in the ambient space can give bound to the distance between their local projections onto the manifold. The proof is provided in Appendix A.
\begin{lemma}\label{lem-lp}
	Suppose Assumption \ref{k} holds. For any given $x,y \in \bar{U}(d)$, where $\bar{U}(d):=\{x\in \mathbb{R}^n: d(x, \mathcal{M}) \leq d\}$, and $d \in (0, \frac{1}{2\kappa_g})$, it holds
	\begin{equation}\label{lp}
	d_\mathcal{M}(\mathcal{P}|_{\Omega}(x), \mathcal{P}|_{\Omega}(y)) \leq \tfrac{2}{1-2\kappa_g d}\|x-y\| .
	\end{equation}
\end{lemma}

Let $\mathcal{X}_{IAM}$ denote the set of IAM defined in \eqref{iam}. The following proposition gives the relationship between $\mathcal{X}_{IAM}$ and $\mathcal{P}_{\mathcal{M}}(\bar X)$ within a neighborhood. 
\begin{proposition}\label{proj}
	For a series of points $X_1, \dots , X_N \in \mathcal{M}$, if their Euclidean average satisfies $d(\bar X, \mathcal{M}) <\frac{1}{2\kappa_g}$, then it holds $\mathcal{P}|_{\Omega}(\bar X) = \hat X,$ 
	where $\hat{X}$ denote the IAM within the $\Omega$-neighborhood defined as $\Omega  = \text{Exp}_{\hat X}(B(0_{\hat X}, D))$.
\end{proposition}
	\begin{proof}
		According to Lemma \ref{nv}, since $d(\bar X, \mathcal{M}) <\frac{1}{2\kappa_g}$, $\mathcal{P}|_{\Omega}(\bar X)$ is unique. It holds, 
		\begin{align*}
		\mathcal{P}|_{\Omega}(\bar X) =& \arg \min_{Y\in \Omega}\|Y - \bar{X}\| = \arg \min_{Y\in \Omega}\|Y - \bar{X}\|^2\\
		= & \arg \min_{Y\in \Omega} (\|Y\|^2 - 2\langle Y, \bar X \rangle).
		\end{align*}
		For any fixed $X_1, \dots , X_N \in \mathcal{M}$, by removing the constant terms, we have 
		$\hat X= \arg \min_{Y \in \Omega} \sum_{i=1}^N \|Y - X_i\|^2 = \arg \min_{Y \in \Omega} (\|Y\|^2 - 2\langle Y, \bar X \rangle).$
		
		Thus, within the neighborhood $\Omega$, we have $\hat X = \mathcal{P}|_{\Omega}(\bar X)$.
	\end{proof}
	
	Based on Lemma \ref{lem-lp} and Proposition \ref{proj}, the relationship of distance between IAMs and the corresponding Euclidean means is given in the next lemma.
	\begin{corollary}\label{e-r ave}
		Let $\hat X,\hat Y$ denote the IAM and $\bar X,\bar Y$ denote the Euclidean average.  If $\bar X,~\bar{Y} \in \bar{U}(\frac{1}{6\kappa_g})$, then $\|\hat Y - \hat X\| \leq d_\mathcal{M}(\hat{X}, \hat{Y}) \leq 3 \|\bar Y-\bar X\|.$
	\end{corollary}
	
	\subsection{Consensus Analysis}
	In this subsection, we establish a local linear convergence result for the consensus error. We begin by introducing the corresponding local region. Assuming that $\delta_1 \leq \frac{1}{20\kappa_g}$, $\delta_2 \leq \frac{\delta_1}{5}$ and recalling $D =  \frac{2-\sqrt{2}}{\kappa_g}$ defined in Lemma \ref{nv}, we denote
	\begin{align}\label{l_r}
	\mathcal{S}&:=\mathcal{S}_1 \bigcap\mathcal{S}_2 \bigcap \mathcal{S}_3,\notag\\
	\mathcal{S}_1&:=\{\bold{X} \in \mathcal{M}^N: \max_{i}\|X_i- \hat{X}\| \leq \delta_1\},\\
	\mathcal{S}_2&:=\{\bold{X}\in \mathcal{M}^N:\|\bold{X}-\hat {\bold{X}}\|^2 \leq N \delta_2^2\},\notag\\
	\mathcal{S}_3&:=\{\bold{X} \in \mathcal{M}^N:\max_{i}d_{\mathcal{M}}(X_i, \hat{X})\leq D\}.\notag
	\end{align}
	
	Building upon Lemma~\ref{nv}, we extend several useful inequalities established for the Stiefel manifold in \cite{chen2023local, wang2023} to general embedded submanifolds. The corresponding inequalities are stated in Lemmas~\ref{ave}--\ref{b-v}, and their proofs are deferred to the Supplementary of \cite{zhao2025distributed} due to space limitations.
	
	First, we show that the distance between $\bar X$ and $\hat X$ can be  bounded by the consensus error scaled with $\kappa_g$. 
	\begin{lemma}\label{ave}
		Suppose Assumption \ref{k} holds. If for any $i \in \mathcal{N}$, $d_\mathcal{M} (X_i, \hat{X}) \leq D$ where $D$ is defined in Lemma \ref{nv}, then there holds$\|\bar X - \hat X\| \leq \frac{\kappa_g}{N} \|\bm{X}-\hat{\bm{X}}\|^2.$
	\end{lemma}
	This lemma together with the condition $\mathbf{X}_k \in \mathcal{S}_2$ yields the conservative estimate $\|\bar X_k-\hat X_k\| \leq \frac{\kappa_g}{6}$, which guarantees that $\bar X_k$ remains sufficiently close to $\mathcal M$ for the distance-dependent assumptions to hold.
	
	Next, we estimate the Riemannian gradient of the consensus potential $h_t(\bm{X})$ which will be used in the consensus analysis. 
	\begin{lemma}\label{g}
		Suppose Assumption \ref{assu-3.2} and \ref{k} hold. Denote by $L_t:=1-\sigma_N(W^t)$, where $\sigma_N(W^t)$ represents the smallest eigenvalue of $W^t$. Then $L_t\in (0,2]$ and for any $X_i \in \mathcal{S}_3,~\forall i$, there holds
		\begin{itemize}
			\item[(\romannumeral1).]$\|\sum_{i=1}^N \text{grad} \ h_{i,t}(\bold{X})\| \leq 2\kappa_g L_t\|\bold{X}- \bold{\hat X}\|^2$;
			\item[(\romannumeral2).] $\|\text{grad}\ h_t(\bm{X})\| \leq L_t\|\bm{X} - \hat{\bm{X}}\|$;
			\item[(\romannumeral3).] If it holds $\bm{X} \in \mathcal{S}_1$, then $\max_{i\in \mathcal{N}}\|\text{grad}\ h_{i,t}(\bm{X})\| \leq 2 \delta_1$.
		\end{itemize}
	\end{lemma}
	
	\begin{lemma}\label{rs}
		Suppose Assumption \ref{assu-3.2}, \ref{k} hold. If $\bold{X} \in \mathcal{S}$, it holds
		$\langle \bold{X}-\bold{\hat X}, \text{grad}\ h_t(\bold{X}) \rangle \geq \frac{\Phi_{\kappa_g}}{2L_t} \|\text{grad}\ h_t(\bold{X})\|^2$.
		where $\Phi_{\kappa_g}: = 2-4\kappa_g \max_{i}\|X_i - \hat X\|$ and $\Phi_{\kappa_g} \in [\frac{9}{5}, 2]$.
	\end{lemma}
	
	We make the following assumption on $\beta_k$.
	\begin{assumption}\label{beta}
		The positive sequence $\{\beta_k\}_{k\geq 0}$ is nondecreasing and satisfies$\sum_{k=0}^{\infty}\frac{1}{\beta_k}=\infty,~\lim_{k \rightarrow\infty}\beta_k=\infty,~\lim_{k \rightarrow\infty} \frac{\beta_{k+1}}{\beta_k}=1.$
	\end{assumption}
	
	For weakly-convex functions in Euclidean space, there is a useful inequality known as the {\it  subgradient inequality}, i.e., $\phi(Y)\geq \phi(X)+\langle g_X, Y-X \rangle - \frac{\rho}{2}\|Y-X\|^2,~g_X \in \partial \phi(X),~X,Y\in \mathbb{R}^n$. We extend this to weakly-convex functions on $\mathcal{M}$ based on Lemma \ref{nv}.
	\begin{lemma}\label{R-sub}
		Suppose Assumption \ref{k} holds. If function $\phi(\cdot)$ is $L$-Lipschitz continuous on $\mathcal{M}$ and $\rho$-weakly-convex, then for any fixed $X\in \mathcal{M}$, when $Y \in \Omega$, we have $
		\phi(Y) \geq \phi(X)+\langle g_X, Y - X \rangle - \frac{\rho+2\kappa_g L}{2}\|Y-X\|^2,$
		where $g_X \in \partial_R \phi(X)$.
	\end{lemma}
	
	By the first-order optimality condition of \eqref{prox}, we derive the following lemma, which shows that proximal algorithm framework is a backward algorithm. 
	\begin{lemma}\label{b-v}
		Suppose Assumption \ref{k} holds. Denote $\tilde{X}_{i,k} = X_{i,k}+v_{i,k}$. The search direction $v_{i,k}$ satisfies $0 = \beta_k v_{i,k} + u_{i,k}$, where $u_{i,k} \in \partial_R F_{i,X_{i,k}}(\tilde{X}_{i,k}, \xi_{i,k})$, and can be bounded by $\|v_{i,k}\| \leq \frac{L}{\beta_k}$ for $\forall i\in \mathcal{N}, k\geq 0$. 
	\end{lemma}
	
	The following lemma establishes a recursion relating the consensus errors at iterations $k$ and $k+1$.
	\begin{lemma}\label{c-e}
		Suppose that all the assumptions hold. Consider the iterates generated by Algorithm \ref{alg:1}, where the step sizes $0<\alpha < \bar{\alpha}:=\min\{ \frac{\Phi_{\kappa_g}}{4L_t(M_1+M_2\delta_1)}\}$ with $L_t$ and $\Phi_{\kappa_g}$ defined in Lemma \ref{g} and Lemma \ref{rs}, and $\beta_k \geq \rho$. If $\bold{X}_k \in \mathcal{S}$, it holds
		\begin{align}\label{c-err}
		\|\bold{X}_{k+1} - \bold{\hat X}_{k+1}\|^2 \leq  \rho_t^2\| \bm{X}_k - \bold{\hat X}_{k}\|^2 +\tfrac{C_1^2NL^2}{\beta_k^2},
		\end{align}
		where $\rho_t:=\sqrt{1-(\alpha \Phi_{\kappa_g} L_t -4\alpha^2L_t^2(M_1+M_2\delta_1) )} <1$, and $C_1^2:=2(M_1+2M_2\delta_1)+\frac{1}{2\alpha^2 L_t^2 M_1}$. Here, $M_1,~M_2$ are parameters associated with the retraction in Lemma \ref{lem-r}.
	\end{lemma}
	\begin{proof}
		Since IAM is the minima of \eqref{iam}, we have\begin{align}\label{c_1}
		&\|\bold{X}_{k+1} - \bold{\hat X}_{k+1}\|^2 \leq \|\bold{X}_{k+1} - \bold{\hat X}_{k}\|^2 	=\|\bold{X}_{k+1} - \bm{X}_k \|^2\notag\\
		+&\| \bm{X}_k - \bold{\hat X}_{k}\|^2+2\langle \bold{X}_{k+1} - \bm{X}_k, \bm{X}_k - \bold{\hat X}_{k}\rangle.
		\end{align}
		The following proof intends to give bound to the three terms on the RHS. First, by using \eqref{R-1}, we derive
		\begin{align}\label{c_2}
		\|\bold{X}_{k+1} - \bm{X}_k \|^2 \leq M_1 \|\alpha \text{grad}\ h_{t}(\bold{X}_k) - \bold{v}_k\|^2.
		\end{align}
		Then, for the inner product term, by utilizing Lemma \ref{lem-r} and Lemma \ref{rs}, it follows that 
		\begin{align}\label{c_3}
		&\langle \bold{X}_{k+1} - \bm{X}_k, \bm{X}_k - \bold{\hat X}_{k}\rangle \leq \left(2M_2\delta_1+\tfrac{1}{4\alpha^2 L_t^2 M_1}\right) \|\bold{v}_k\|^2 \\
		+& (\tfrac{4\alpha^2L_tM_2\delta_1-\alpha \Phi_{\kappa_g}}{2L_t})\|\text{grad}\ h_t(\bm{X}_k)\|^2
		+\alpha^2L_t^2M_1 \|\bold{\hat X}_k - \bm{X}_k\|^2.\notag
		\end{align}
		Substituting \eqref{c_2} and \eqref{c_3} into \eqref{c_1}  yields
		\begin{align*}
		&\|\bold{X}_{k+1} - \bold{\hat X}_{k+1}\|^2 \leq (1+2\alpha^2L_t^2M_1)\| \bm{X}_k - \bold{\hat X}_{k}\|^2\\
		+& (2\alpha^2(M_1+2M_2\delta_1)-\tfrac{\alpha \Phi_{\kappa_g}}{L_t})\|\text{grad}\ h_t(\bm{X}_k)\|^2\\
		+& \left(2(M_1+2M_2\delta_1)+\tfrac{1}{2\alpha^2 L_t^2 M_1}\right) \|\bold{v}_k\|^2.
		\end{align*}
		By using Lemma \ref{g}-(\romannumeral2) and Lemma \ref{b-v}, we finally prove \eqref{c-err}, and $\rho_t < 1$ since $\alpha < \frac{\Phi_{\kappa_g}}{4L_t(M_1+M_2\delta_1)}$. 
	\end{proof}
	
	To justify the local analysis, we next show that the iterates $\boldsymbol{X}_k$ remain in the set $\mathcal{S}$ if the initialization lies in $\mathcal{S}$, where $\mathcal{S}$ is defined in \eqref{l_r}. The proof is deferred to Appendix~B. 
	\begin{theorem}\label{th-1}
		Suppose that all the assumptions hold. Consider the iterates generated by Algorithm \ref{alg:1}, where the step sizes $0<\alpha < \bar{\alpha}:=\min\{1,\frac{\kappa_g}{M_2}, \frac{\Phi_{\kappa_g}}{4L_t(M_1+M_2\delta_1)}\}$ with $L_t$ and $\Phi_{\kappa_g}$ defined in Lemma \ref{g} and Lemma \ref{rs}, $\beta_k \geq \max\{\frac{5L}{\alpha\delta_2}, \frac{LC_1}{\delta_2\sqrt{1-\rho_t^2}} \}$, and $t \geq \lceil \log_{\sigma_2}(\frac{1}{5\sqrt{N}})\rceil$. If the initial points satisfy $\bold{X}_0 \in\mathcal{S}$, then we have $\bold{X}_k \in \mathcal{S}$ for any $ k \geq 0$.
	\end{theorem}
	
	\begin{remark}\label{dis_s}
			Theorem~\ref{th-1} plays a crucial role in the subsequent consensus and convergence, as it guarantees that the local condition required in the key lemmas and technical lemmas are satisfied throughout the iteration process once the initialization $\mathbf{X}_0$ lies within a local region $\mathcal{S}$. This means, the initial points across different agents are chosen to be sufficiently close. This requirement is consistent with the initialization strategies adopted in many existing distributed optimization works (see e.g. \cite{pmlr-v139-chen21g, neal2011distributed}), which guide us to initialize all agents from a common starting point in practice, especially for the machine learning applications, such as principal component analysis, dictionary learning.
		\end{remark}
		
		The condition $t \geq \lceil \log_{\sigma_2}(\frac{1}{5\sqrt{N}})\rceil$ is necessary to ensure that $\boldsymbol{X}_k \in \mathcal{S}_1$ for $k \geq 0$. For some specific networks, we estimate the number of $t$ by the number of agents $N$ as shown in Table \ref{tab:t}, since $\sigma_2$ can be determined by $N$ when $W$ is the lazy Metropolis matrix \cite{nedic2017achieving}.
		\begin{table}[ht]
			\centering
			\setlength{\abovecaptionskip}{0cm}
			\setlength{\belowcaptionskip}{-0.2cm}
			\caption{Lower bounds of $t$ for certain networks.}
			\begin{tabular}{ccccc}
				\toprule
				Graph&Path&Ring&ER&Complete  \\
				\midrule
				$t$& $\mathcal{O}(N^2\log N)$ & $\mathcal{O}(N^2\log N)$ &$\mathcal{O}(\log N)$&$1$\\
				\bottomrule
			\end{tabular}
			\label{tab:t}
		\end{table}
		
		Subsequently, we give the convergence rate and error bound of consensus. The proof is left in Appendix C.
		\begin{theorem}\label{th-c}
			Suppose that all the assumptions hold. Consider the iterates generated by Algorithm \ref{alg:1}, where $\alpha,~\beta_k,~t$ satisfies the same condition as Theorem \ref{th-1}. If the initial points satisfy $\bold{X}_0 \in\mathcal{S}$, then for any $  k \geq 0$, 
			
			(\romannumeral1). the consensus error satisfies 
			\begin{equation*}
			\|\bold{X}_{k+1} - \bold{\hat X}_{k+1}\|^2 \leq \rho_t^{2k} \| \bm{X}_0 - \bold{\hat X}_{0}\|^2 +NL^2C_1^2 \sum\nolimits_{l=0}^{k}\tfrac{\rho_t^{2l}}{\beta_{k-l}^2},
			\end{equation*} where $\rho_t\in (0,1)$ and $C_1^2$ are defined in Lemma \ref{c-e};
			
			(\romannumeral2). there holds $\frac{1}{N} \| \bm{X}_k - \bold{\hat X}_{k}\|^2 \leq \frac{CL^2}{\beta_k^2}$ where $C = (\tilde{C}+\frac{4C_1^2}{1-\rho_t^2}) = \mathcal{O}(\frac{1}{1-\rho_t^2})$, which is independent of $L$.
		\end{theorem}
		
		Theorem \ref{th-c} (\romannumeral2) shows that the iterates $\bold{X}_k$ achieve consensus linearly to a $\mathcal{O}(\frac{1}{\beta_k^2})$-neighborhood of $\hat{\bold{X}}_k$.
		
		\begin{remark}
			The coefficient $\rho_t$ characterizes the consensus rate on the embedded submanifolds: smaller values of $\rho_t$ correspond to faster consensus. When the manifold degenerate into Euclidean space, we have $M_1=1$, $M_2=0$, and $\Phi_{\kappa_g} = 2$, yielding the smallest rate factor $\rho_t = \sqrt{1-(2\alpha L_t - 4\alpha^2 L_t^2)}$. Therefore, the Euclidean case represents the most favorable scenario for consensus, while the curvature-dependent terms introduce additional geometric effects that enlarge the contraction factor $\rho_t$ and lead to a slower consensus rate.
		\end{remark}
		
		\subsection{Convergence Results in the Local Region}
		We now establish the convergence rate of Algorithm \ref{alg:1} in local region $\mathcal{S}$. 
		
		In advance, we give some preparatory lemmas that will be used in the convergence analysis. The following lemma indicates that the proximal map is single-valued and satisfies a Lipschitz-type inequality. The proof is derived following a similar idea of \cite[Lemma 4.2]{wang2023}, and can be referred to the Supplementary of \cite{zhao2025distributed} for a complete expression.
		\begin{lemma}\label{P-lip}
			Suppose Assumption \ref{assu-3.1} holds and let $\lambda \in (0, \frac{1}{\rho+6\kappa_g L})$. The proximal map $P_{\lambda f}(X)$ is single-valued and satisfies a Lipschitz-type inequality
			$\|P_{\lambda f}(X) - P_{\lambda f}(Y)\|\leq (1-\lambda(\rho+6\kappa_g L))^{-1}\|X-Y\|.$
		\end{lemma}
		
		Since the cost function is nonsmooth, we are not able to show that there is a sufficient descent of the original cost function as in \cite{wang2024proxtrack, deng2023decentralized}. Instead, we consider the average of the Moreau envelope defined in \eqref{Me} and show the one-step improvement of Algorithm \ref{alg:1}.
		\begin{theorem}\label{th-2}
			Suppose that all the assumptions hold. Denote $\gamma:=\rho+2\kappa_g L$.  Consider the Algorithm \ref{alg:1}, where $0<\alpha < \bar{\alpha}$ defined in Theorem \ref{th-1}, $\beta_k \geq \bar{\beta}= \max\{\frac{5L}{\alpha\delta_2}, \frac{C_1L}{\delta_2\sqrt{1-\rho_t^2}} \}$, and $t \geq \lceil \log_{\sigma_2}(\frac{1}{5\sqrt{N}})\rceil$. Let $\lambda \in (0, \nu),~\nu:=\min\{\frac{1}{\rho+6\kappa_g L}, \frac{1}{2\gamma+\tau)}\} $. Then, we have 
			\begin{align}\label{one-step}
			\mathbb{E}_k[\bar f_\lambda(\bold{X}_{k+1})] \leq& \bar f_\lambda (\bold{X}_{k}) + e(\bold{X}_k, \bold{\hat X}_k) +b_k\\
			+& \tfrac{(\gamma+\tau/2)-\frac{1}{2\lambda}}{\lambda \beta_k}\|\mathcal{P}_{\lambda f}(\hat{X}_{k})-\hat{X}_k\|^2,\notag
			\end{align}
			where $\bar f_\lambda(\bold{X}_k): = \frac{1}{N}\sum_{i=1}^{N}f_\lambda(X_{i,k})$, $e(\bold{X}_k, \bold{\hat X}_k)$ is the term related to the consensus error denoted by $e(\bold{X}_k, \bold{\hat X}_k):=\frac{1}{2\lambda}\Big(C_2\alpha^2 L_t^2+ 8\kappa_g \alpha \lambda LL_t+2(1+\tilde{L}^2)(\alpha+\frac{2\gamma+\tau}{\beta_k}) \Big)\frac{\|\bold{X}_k - \bold{\hat X}_k\|^2}{N}+\frac{L(\tilde{L}+1)}{ 2\lambda\beta_k}\frac{\sum_{i=1}^{N}\|X_{i,k}-\hat X_k\|}{N}$, and $b_k = \frac{3L^2}{2\beta_k^2}+\frac{\gamma L^2}{\beta_k^3}$.
		\end{theorem}
		\begin{proof}
			According to \eqref{e-iter}, for each $ i\in \mathcal{N}$, it holds 
			\begin{align}\label{t1}
			&\|X_{i,k+1} - P_{\lambda f} (X_{i,k})\|^2\notag \\
			=& \|\mathcal{R}_{X_{i,k}}\left(-\alpha \text{grad}~h_{i,t}(\bold{X}_k)+v_{i,k} \right) - P_{\lambda f} (X_{i,k})\|^2\\
			\leq & \|X_{i,k} - P_{\lambda f} (X_{i,k})\|^2+M_1^2\|\alpha \text{grad}~h_{i,t}(\bold{X}_k)-v_{i,k}\|^2 \notag\\
			+&2\langle \mathcal{R}_{X_{i,k}}\left(-\alpha \text{grad}~h_{i,t}(\bold{X}_k)+v_{i,k} \right) - X_{i,k}, X_{i,k} - P_{\lambda f} (X_{i,k})\rangle. \notag
			\end{align}
			where the inequality is derived by using \eqref{R-1}.

			By the definition of proximal point $P_{\lambda f}(X)$, we have $f(X) \geq f(P_{\lambda f}(X))+\frac{1}{2\lambda}\|X - P_{\lambda f}(X)\|^2$. Since $f$ is $L$-Lipschitz, for $\forall X \in \mathcal{M}$, we derive that 
			\begin{align}\label{prox-l}
			\|X - P_{\lambda f}(X)\| \leq 2\lambda L.
			\end{align}
			Thus, by \eqref{R-2} and  \eqref{prox-l}, the last term of \eqref{t1} can be further derived as
			\begin{align}\label{t2}
			&\langle \mathcal{R}_{X_{i,k}}\left(-\alpha \text{grad}~h_{i,t}(\bold{X}_k)+v_{i,k} \right) - X_{i,k}, X_{i,k} - P_{\lambda f} (X_{i,k})\rangle \notag\\
			\leq &  2\lambda L M_2 \|\alpha \text{grad}~h_{i,t}(\bold{X}_k)-v_{i,k}\|^2\\
			+& \langle -\alpha \text{grad}~h_{i,t}(\bold{X}_k)+v_{i,k} , X_{i,k} - P_{\lambda f} (X_{i,k})\rangle.\notag
			\end{align}
			
			By combining \eqref{t1} and \eqref{t2}, and taking the summation from $i = 1$ to $N$, we obtain the summation of one-step error to Moreau envelop \begin{align}\label{t3}
			&\sum_{i=1}^{N} \|X_{i,k+1} - P_{\lambda f} (X_{i,k})\|^2\leq \sum_{i=1}^{N}  \|X_{i,k} - P_{\lambda f} (X_{i,k})\|^2\notag\\
			+&\underbrace{(M_1^2+4\lambda L M_2)\sum\nolimits_{i=1}^{N} \|\alpha \text{grad}~h_{i,t}(\bold{X}_k)-v_{i,k}\|^2}_{term 1} \\
			+& \underbrace{2\sum\nolimits_{i=1}^{N}  \langle -\alpha \text{grad}~h_{i,t}(\bold{X}_k), X_{i,k} - P_{\lambda f} (X_{i,k})\rangle}_{term 2}\notag\\
			+&\underbrace{2\sum\nolimits_{i=1}^{N}  \langle v_{i,k} , X_{i,k} - P_{\lambda f} (X_{i,k})\rangle}_{term 3}.\notag
			\end{align}
			
			We then give bound to the above terms. For $term1$, by utilizing (\romannumeral2) of Lemma \ref{g} and Lemma \ref{b-v}, it follows that 
			\begin{align*}
			\sum_{i=1}^{N} \|\alpha \text{grad}~h_{i,t}(\bold{X}_k)-v_{i,k}\|^2 \leq &2 \alpha^2 \|\text{grad}~h_{t}(\bold{X}_k)\|^2 + 2\|\bold{v}_k\|^2  \\
			\leq   &2 \alpha^2 L_t^2\|\bold{X}_k - \bold{\hat X}_k\|^2 + \tfrac{2NL^2}{\beta_k^2}.
			\end{align*} Thus, we have  $term1 \leq (M_1^2+4\lambda L M_2)\alpha^2 L_t^2\|\bold{X}_k - \bold{\hat X}_k\|^2 + \frac{2L^2N(M_1^2+4\lambda L M_2)}{\beta_k^2}.$

			For $term2$, by Cauchy inequality and \eqref{prox-l}, we derive that  \begin{align*}
			&term2 = 2\alpha\sum_{i=1}^N \langle  \text{grad}~h_{i,t}(\bold{X}_k), P_{\lambda f} (X_{i,k}) - X_{i,k}\rangle \\
			=&2\alpha \sum_{i=1}^N \langle  \text{grad}~h_{i,t}(\bold{X}_k), P_{\lambda f} (\hat X_k) - \hat X_k \rangle \notag\\
			+&2\alpha \sum_{i=1}^N\langle  \text{grad}~h_{i,t}(\bold{X}_k), \hat X_k - X_{i,k} + P_{\lambda f} (X_{i,k}) -  P_{\lambda f} (\hat X_k) \rangle\\
			\leq & 4\alpha \lambda L \|\sum_{i=1}^N\text{grad}~h_{i,t}(\bold{X}_k)\| + \alpha \|\text{grad}~h_{t}(\bold{X}_k)\|^2\\
			+& \alpha \sum_{i=1}^N \|\hat X_k - X_{i,k} + P_{\lambda f} (X_{i,k}) -  P_{\lambda f} (\hat X_k)\|^2 .
			\end{align*} 
			By Young's inequality and Lemma \ref{P-lip}, $term2$ turns out to be 
			\begin{align}
			term2 \leq & 4\alpha \lambda L \|\sum_{i=1}^N\text{grad}~h_{i,t}(\bold{X}_k)\| +\alpha \|\text{grad}~h_{t}(\bold{X}_k)\|^2\notag\\
			+ &2\alpha(1+\tilde{L}^2)\|\bold{X}_k - \bold{\hat X}_k\|^2.
			\end{align}
			Then, by Lemma \ref{g}, we have $term2 \leq  (8\kappa_g \alpha \lambda LL_t + 2\alpha(1+\tilde{L}^2)+\alpha L_t^2)\|\bold{X}_k - \bold{\hat X}_k\|^2$, where $\tilde{L}:= (1-\lambda(\rho+6\kappa_g L))^{-1}$.
			
			Next, we consider $term3$. The idea to bound $term3$ follows by the steps below. We first analyze $\langle v_{i,k} , X_{i,k} - P_{\lambda f} (X_{i,k})\rangle$ for $\forall i$, and then bound it by the stationary measurement terms and terms related to $\beta_k$. Finally, we sum the inequality up to derive $term3$ in expectation.
			
			Let $\tilde{X}_{i,k}:=X_{i,k}+v_{i,k}$. By Lemma \ref{b-v}, we implies that \begin{equation}\label{t_3}
			\langle v_{i,k} , X_{i,k} - P_{\lambda f} (X_{i,k})\rangle = \tfrac{1}{\beta_k}\langle u_{i,k} , P_{\lambda f} (X_{i,k}) - X_{i,k}\rangle,
			\end{equation} where $u_{i,k} \in \partial_R F_{i,X_{i,k}}(\tilde{X}_{i,k+1}, \xi_{i,k})$. According to the Riemannian subgradient inequality in Lemma \ref{R-sub}, we have 
			\begin{align}\label{t5}
			&F_{i,X_k}(\mathcal{P}_{\lambda f}(X_{i,k}), \xi_{i,k}) - F_{i,X_k}(\tilde{X}_{i,k}, \xi_{i,k})\\
			\geq& \langle u_{i,k}, \mathcal{P}_{\lambda f}(X_{i,k}) - \tilde{X}_{i,k} \rangle - \tfrac{\rho+2\kappa_g L}{2}\|\mathcal{P}_{\lambda f}(X_{i,k}) - \tilde{X}_{i,k}\|^2.\notag
			\end{align}
			Let $\gamma:=\rho+2\kappa_g L$. By substituting \eqref{t_3} into \eqref{t5} and rearranging terms, we get \begin{align*}
			&\langle v_{i,k} , X_{i,k} - P_{\lambda f} (X_{i,k})\rangle\\
			\leq &\frac{1}{\beta_k}\Big(F_{i,X_k}(\mathcal{P}_{\lambda f}(X_{i,k}), \xi_{i,k}) - F_{i,X_k}(\tilde{X}_{i,k}, \xi_{i,k})\\
			+&\gamma\|X_{i,k} -\mathcal{P}_{\lambda f}(X_{i,k}) \|^2 +  (\gamma+\frac{\beta_k}{2})\|v_{i,k}\|^2\Big)\notag
			\end{align*} 
			
			Denote by $\mathbb{E}_{k}[\cdot]$ the conditional expectation $\mathbb{E}[\cdot|\mathcal{F}_k]$, where $\mathcal{F}_k:=\{\xi_{i,t},\forall i \in \mathcal{N},~\text{and}~0\leq t\leq k-1\}$.
			Since $F_{i,X}(\cdot, \xi_i)$ is $L$-Lipschitz continuous and by Definition \ref{s-model}(b), first we bound
			\begin{align}\label{t4}
			&\mathbb{E}_{k}[F_{i,X_k}(\mathcal{P}_{\lambda f}(X_{i,k}), \xi_{i,k}) - F_{i,X_k}(\tilde{X}_{i,k}, \xi_{i,k})] \notag\\		
			\leq&\mathbb{E}_{k}[F_{i,X_k}(\mathcal{P}_{\lambda f}(X_{i,k}), \xi_{i,k}) - F_{i,X_k}({X}_{i,k}, \xi_{i,k})] \\
			+& L\mathbb{E}_{k}\|X_{i,k} - \tilde{X}_{i,k}\|\notag\\
			\leq &f_i(\mathcal{P}_{\lambda f}(X_{i,k}) - f_i(X_{i,k})+\frac{\tau}{2}\|\mathcal{P}_{\lambda f}(X_{i,k}) - X_{i,k}\|^2\notag\\
			+& L\mathbb{E}_{k}[\|v_{i,k}\|]. \notag
			\end{align}
			
			By taking conditional expectation of \eqref{t_3} and substituting \eqref{t4} into it, we estimate the expectation of $i$-th item in {\it term3}
			\begin{align*}
			&\mathbb{E}_{k}[\langle v_{i,k} , X_{i,k} - P_{\lambda f} (X_{i,k})\rangle]\\
			\leq&\tfrac{1}{\beta_k} \Big(f_i(\mathcal{P}_{\lambda f}(X_{i,k})) - f_i(X_{i,k})+ \tfrac{2\gamma+\tau}{2}\|\mathcal{P}_{\lambda f}(X_{i,k}) - X_{i,k}\|^2\notag\\
			+&L\mathbb{E}_{k}[\|v_{i,k}\|]+(\frac{\beta_k}{2}+\gamma)\mathbb{E}_{k}[\|v_{i,k}\|^2]\Big).
			\end{align*} 
			By Lemma \ref{b-v}, it follows that
			\begin{align}\label{t6}
			&\mathbb{E}_{k}[\langle v_{i,k} , X_{i,k} - P_{\lambda f} (X_{i,k})\rangle] \leq\tfrac{1}{\beta_k}\Big( f_i(\mathcal{P}_{\lambda f}(X_{i,k})) - f_i(X_{i,k})\notag\\
			+& \tfrac{2\gamma+\tau}{2}\|\mathcal{P}_{\lambda f}(X_{i,k}) - X_{i,k}\|^2+\tfrac{3L^2}{2\beta_k}+\tfrac{\gamma L^2}{\beta_k^2}\Big).
			\end{align}
			According to Lemma \ref{P-lip}, we give bound to the function value differences
			\begin{align}\label{f_i}
			&f_i(\mathcal{P}_{\lambda f}(X_{i,k})) - f_i(X_{i,k})\notag\\
			= & f_i(\mathcal{P}_{\lambda f}(X_{i,k})) -f_i(\mathcal{P}_{\lambda f}(\hat{X}_{k}))+f_i(\mathcal{P}_{\lambda f}(\hat{X}_{k})) -f_i(\hat{X}_{k})\notag\\
			+& f_i(\hat{X}_{k}) - f_i(X_{i,k})\notag\\
			\leq & L\|\mathcal{P}_{\lambda f}(X_{i,k}) -\mathcal{P}_{\lambda f}(\hat{X}_{k})\|+f_i(\mathcal{P}_{\lambda f}(\hat{X}_{k})) -f_i(\hat{X}_{k}) \\
			+& L\|X_{i,k}-\hat X_k\|\notag\\
			\leq & L(\tilde{L}+1)\|X_{i,k}-\hat X_k\|+f_i(\mathcal{P}_{\lambda f}(\hat{X}_{k})) -f_i(\hat{X}_{k}),\notag
			\end{align}
			and the proximity difference 
			\begin{align}\label{p_i}
			&\|\mathcal{P}_{\lambda f}(X_{i,k}) - X_{i,k}\|^2\\
			\leq & \|\mathcal{P}_{\lambda f}(X_{i,k}) -\mathcal{P}_{\lambda f}(\hat{X}_{k})+\mathcal{P}_{\lambda f}(\hat{X}_{k})-\hat{X}_k+\hat{X}_k- X_{i,k}\|^2\notag\\
			\leq & 2\|\mathcal{P}_{\lambda f}(\hat{X}_{k})-\hat{X}_k\|^2+4(1+\tilde{L}^2)\|X_{i,k}-\hat X_k\|^2.\notag
			\end{align}
			Plugging \eqref{f_i} and \eqref{p_i} into \eqref{t6} implies that
			\begin{align}\label{t7}
			&\mathbb{E}_{k}[\langle v_{i,k} , X_{i,k} - P_{\lambda f} (X_{i,k})\rangle]\notag\\
			\leq&\Big( L(\tilde{L}+1)\|X_{i,k}-\hat X_k\|+f_i(\mathcal{P}_{\lambda f}(\hat{X}_{k})) -f_i(\hat{X}_{k}) \notag\\
			+& (2\gamma+\tau)\|\mathcal{P}_{\lambda f}(\hat{X}_{k})-\hat{X}_k\|^2 \notag\\
			+& 2(2\gamma+\tau)(1+\tilde{L}^2)\|X_{i,k}-\hat X_k\|^2+\frac{3L^2}{2\beta_k}+\frac{\gamma L^2}{\beta_k^2} \Big).
			\end{align}
			
			By taking summation of \eqref{t7} from $i=1$ to $N$ and since$f(\hat{X}_{k}) \ge f(\mathcal{P}_{\lambda f}(\hat{X}_{k}))+ \frac{1}{2\lambda} \|\mathcal{P}_{\lambda f}(\hat{X}_{k})-\hat{X}_k\|^2$, it holds
			\begin{align*}
			&\mathbb{E}_k[\sum_{i=1}^{N} term3] \leq \tfrac{1}{\beta_k}\Big(L(\tilde{L}+1)\sum_{i=1}^{N}\|X_{i,k}-\hat X_k\|\\
			+&2(2\gamma+\tau)(1+\tilde{L}^2)\|\bold{X}_{k}-\hat {\bold{X}}_k\|^2\notag\\
			+&N(2\gamma+\tau-\tfrac{1}{2\lambda})\|\mathcal{P}_{\lambda f}(\hat{X}_{k})-\hat{X}_k\|^2+\tfrac{3NL^2}{2\beta_k}+\tfrac{N\gamma L^2}{\beta_k^2}\Big).
			\end{align*}
			
			According to the definition of Moreau envelope, we have $
			\mathbb{E}_k[f_\lambda(X_{i,k+1})] \leq f(P_{\lambda f} (X_{i,k})) + \frac{1}{2\lambda}\mathbb{E}_k[\|X_{i,k+1} - P_{\lambda f} (X_{i,k})\|^2].$
			Denote $\bar f_\lambda(\bold{X}_k): = \frac{1}{N}\sum_{i=1}^{N}f_\lambda(X_{i,k})$.  By summing \eqref{t3} from $i=1$ to $N$ and plugging in the bounds of terms $t_1,~t_2,~t_3$, we have
			\begin{align*}
			&\mathbb{E}_k[\bar f_\lambda(\bold{X}_{k+1})]\notag\\
			\leq &\tfrac{1}{N} \sum\nolimits_{i=1}^{N} \left(f(P_{\lambda f} (X_{i,k})) +\tfrac{1}{2\lambda} \mathbb{E}_k[\|X_{i,k+1} - P_{\lambda f} (X_{i,k})\|^2]\right) \notag\\
			\leq &\bar f_\lambda (\bold{X}_{k}) + e(\bold{X}_k, \bold{\hat X}_k) +b_k+ \tfrac{(\gamma+\tau/2)-\frac{1}{2\lambda}}{\lambda \beta_k}\|\mathcal{P}_{\lambda f}(\hat{X}_{k})-\hat{X}_k\|^2
			\end{align*}where $e(\bold{X}_k, \bold{\hat X}_k)$ and $b_k$ are denoted in Theorem \ref{th-2}.
		\end{proof}

		Based on Theorem \ref{th-2}, we further develop the convergence rate of Algorithm \ref{alg:1} using the stationary measurement
		
		\begin{theorem}\label{th-3}
			Consider Algorithm \ref{alg:1}, and suppose that the conditions of Theorem \ref{th-2} hold. Thus,
			\begin{align}\label{rec}
			&\min_{k=0,1,\dots, K} \tfrac{1}{\lambda^2}\mathbb{E}[\|\mathcal{P}_{\lambda f}(\hat{X}_{k})-\hat{X}_k\|^2] \\
			\leq &\frac{2}{1-2\lambda(\gamma+\tau/2)} \frac{\bar{f}_\lambda  (\bold{X}_{0}) - \bar f_\lambda^*+ \sum_{k=0}^{K} a_k +\sum_{k=0}^{K}b_k}{\sum_{k=0}^{K} \frac{1}{\beta_k}} \notag
			\end{align}with $a_k = \frac{CL^2(C_2(2L^2+\alpha^2 L_t^2)+ 8\kappa_g \alpha \lambda LL_t+2\alpha(1+\tilde{L}^2))}{2\lambda\beta_k^2}+\frac{CL^2(1+\tilde{L}^2)(\gamma+\tau/2)}{\lambda \beta_k^3}+\frac{\sqrt{C}L^2(\tilde{L}+1)}{2 \lambda\beta_k^2} $
			Moreover, if let $\lambda \le \frac{1}{2\nu}$ and $\beta_k =\bar{\beta}\sqrt{K+1}$ with $\bar \beta$ defined in Theorem \ref{th-2}, then the order of convergence rate is $ \mathcal{O}\left((1+\kappa_g)(\frac{1}{\sqrt{K+1}}+\frac{1}{K+1})\right).$
		\end{theorem}
		\begin{proof}
			We prove this based on Theorem \ref{th-2}.
			
			Rearranging the terms and taking unconditional expectation of \eqref{one-step} produces \begin{align}\label{t-E}
			&\tfrac{\frac{1}{2\lambda}-(\gamma+\tau/2)}{\lambda \beta_k}\mathbb{E}[\|\mathcal{P}_{\lambda f}(\hat{X}_{k})-\hat{X}_k\|^2] \\
			\leq& \mathbb{E}[\bar f_\lambda  (\bold{X}_{k}) - \bar f_\lambda(\bold{X}_{k+1})]
			+ \mathbb{E}[ e(\bold{X}_k, \bold{\hat X}_k)] +b_k.\notag
			\end{align}
			
			By taking summation of \eqref{t-E} from $k=0$ to $K$, we have
			\begin{align*}
			&\sum _{k=0}^{K}\frac{\frac{1}{2\lambda}-(\gamma+\tau/2)}{\lambda \beta_k}\mathbb{E}[\|{P}_{\lambda f}(\hat{X}_{k})-\hat{X}_k\|^2]\\ 
			\leq& \bar f_\lambda  (\bold{X}_{0}) - \min_{X\in \mathcal{M}} \bar f_\lambda(X) 
			+ \sum_{k=0}^{K} \mathbb{E}[ e(\bold{X}_k, \bold{\hat X}_k)] +\sum _{k=0}^{K}b_k,
			\end{align*}
			Let $\bar f_\lambda^*:=\min_{X\in \mathcal{M}} \bar f_\lambda(X)$. Utilizing Theorem \ref{th-c} (\romannumeral2) further implies that
			\begin{align}
			&\sum_{k=0}^{K}\frac{\frac{1}{2}-\lambda(\gamma+\tau/2)}{\lambda \beta_k}\mathbb{E}[\|{P}_{\lambda f}(\hat{X}_{k})-\hat{X}_k\|^2] \notag\\
			\leq& \bar f_\lambda  (\bold{X}_{0}) -\bar f_\lambda^* + \sum_{k=0}^{K} a_k+\sum_{k=0}^{K}b_k,
			\end{align} 
			where  $a_k: = \frac{CL^2(C_2(2L^2+\alpha^2 L_t^2)+ 8\kappa_g \alpha \lambda LL_t+2\alpha(1+\tilde{L}^2))}{2\lambda\beta_k^2}+\frac{CL^2(1+\tilde{L}^2)(\gamma+\tau/2)}{\lambda \beta_k^3}+\frac{\sqrt{C}L^2(\tilde{L}+1)}{2 \lambda\beta_k^2} \sim \mathcal{O}(\frac{1}{\beta_k^2})$.
			
			By dividing both sides by $\sum_{k=0}^{K}\frac{\frac{1}{2}-\lambda(\gamma+\tau/2)}{\lambda \beta_k}$, we get \eqref{rec} accordingly.
			
			Moreover, when $\beta_k = \bar{\beta}\sqrt{K+1}$, \eqref{rec} implies that
			\begin{align}\label{bound}
			&\min_{k=0,\dots,K} \tfrac{1}{\lambda^2}\mathbb{E}[\|\mathcal{P}_{\lambda f}(\hat{X}_{k})-\hat{X}_k\|^2]  \leq \tfrac{2\bar{\beta}}{1-\lambda(2\gamma+\tau)}\Big(\tfrac{\bar f_\lambda  (\bold{X}_{0}) - \bar f_\lambda^*}{\sqrt{K+1}}\notag\\
			+&\tfrac{\sqrt{C}L^2(\tilde{L}+1)+3L^2\lambda}{2\lambda\sqrt{K+1}}+CL^2\tfrac{C_2(2L^2+\alpha^2 L_t^2)+ 8\kappa_g \alpha \lambda LL_t+2\alpha(1+\tilde{L}^2)}{2\lambda\sqrt{K+1}}\notag\\
			+& \tfrac{CL^2(1+\tilde{L}^2)(2\gamma+\tau)+2\gamma L^2\lambda}{2\lambda (K+1)} \Big).
			\end{align}
			When the weight parameter further satisfies $\lambda \leq \min \{\frac{1}{2\gamma+\tau}, \frac{1}{2(\rho+6L\kappa_g)}\}$, we get
			$\tilde{L} = \frac{1}{1-\lambda(\rho+6L\kappa_g)}\leq 2,~\text{and}~\frac{1}{1-\lambda(2\gamma+\tau)} \leq 2.$
			Plugging this into \eqref{bound} implies that 
			\begin{align}\label{order}
			\min_{k=0,\dots,K} &\tfrac{1}{\lambda^2}\mathbb{E}[\|\mathcal{P}_{\lambda f}(\hat{X}_{k})-\hat{X}_k\|^2] \leq 4\bar{\beta}\Big(\tfrac{(\bar f_\lambda  (\bold{X}_{0}) - \bar f_\lambda^*)}{\sqrt{K+1}} \notag\\
			+& \tfrac{3L^2(\sqrt{C}+\lambda)}{2\lambda \sqrt{K+1}}+ \tfrac{CL^2C_2(2L^2+\alpha^2 L_t^2)+10CL^2\alpha+4L^2C_2}{2\lambda \sqrt{K+1}}\notag\\
			+&\tfrac{4C\alpha \lambda L_tL^3\kappa_g}{\sqrt{K+1}} + \tfrac{5CL^2(2\rho+4\kappa_g+\tau)+4L^2\lambda(\rho+2\kappa_g)}{(K+1)} \Big)\notag\\
			\leq & \mathcal{O}\left((1+\kappa_g)(\tfrac{1}{\sqrt{K+1}}+\tfrac{1}{K+1})\right).
			\end{align}
			If $K$ is sufficiently large, then the convergence rate will be dominated by the term $\mathcal{O}\left(\frac{1+\kappa_g}{\sqrt{K+1}}\right)$.
		\end{proof}
		
		\begin{remark}
			(1). Theorem \ref{th-3} shows that the number of iterations, which is also called the computation complexity, is $\mathcal{O}(\frac{(1+\kappa_g)^2}{\epsilon^2})$. 
			Therefore, larger geodesic curvature leads to a deterioration of the convergence rate. Conversely, when $\kappa_g$ is small, the manifold is locally closer to its tangent spaces, and the resulting geometric effects on the algorithm become less pronounced, yielding improved convergence performance.
			
			(2). The terms involving $M_1,~M_2,~\kappa_g$ arise from the manifold geometry and quantify the influence of the Riemannian manifold on the convergence bound. When $\mathcal{M} = \mathbb{R}^n$, we have $M_1 = 1$, $M_2 = 0$ and $\kappa_g=0$, corresponding to the absence of geometric distortion. In this case, the convergence error bound is reduced, although the order of convergence remains $\mathcal{O}(\frac{1}{\sqrt{K+1}})$. 
			
			(3). When the cost functions are smooth, it holds that $\|\text{grad}f(\mathcal{P}_{\lambda f}(\hat X_k)\| \leq \frac{1}{\lambda^2}\mathbb{E}[\|\mathcal{P}_{\lambda f}(\hat{X}_{k})-\hat{X}_k\|^2] $. Together with \eqref{rec} in Theorem \ref{th-3}, this implies an iteration complexity of $\mathcal{O}(\frac{1}{\sqrt{K+1}})$ matching the convergence order of DPRGD \cite{deng2023decentralized}. At the same time, the proposed framework accommodates a broader class of weakly convex nonsmooth problems and explicitly reveals how the manifold geometry affects the convergence bound.
		\end{remark}

		\section{Numerical Experiments}\label{sec-5}
		In this section, we evaluate three proposed algorithms, namely, the distributed Riemannian stochastic subgradient (DR-SSG), the distributed Riemannian stochastic proximal point (DR-SPP), the distributed Riemannian stochastic proximal-linear (DR-SPL), and compare them with the existing approaches. In addition to experiments on the Stiefel manifold, we consider distributed blind deconvolution on the sphere, with results reported in \cite{zhao2025distributed} due to space limitations. We further investigate the effect of the geodesic curvature parameter $\kappa_g$ through a distributed generalized eigenvalue problem on the generalized Stiefel manifold. All computations are carried out using the {\it Manopt} toolbox \cite{manopt}.
		
		\subsection{Distributed Orthogonal Sparse Dictionary Learning}
		Dictionary learning problem derives from many machine learning applications, whose goal is to learn a suitable representation for the input dataset to turn it into sparse data. Denote the Stiefel manifold by $\text{St}(n,n):=\{X \in \mathbb{R}^{n \times n}:X^\top X = I_n\}$. Assuming that the observation data $Y$ can be approximated by $Y \approx A S$, where $A$ is the dictionary to be learned, and each column of $S$ is sparse. Consider the distributed orthogonal dictionary learning problem with $N$ agents 
		$\min_{X \in \text{St}(n,n)}~\frac{1}{N}\sum_{i=1}^N \left(\frac{1}{m_i}\sum_{j=1}^{m_i} \|y_{(i,j)}^\top X\|_1\right),$ where $m_i$ is the data size and $y_{(i,j)}$ denote the $j$-th column vector of the data $Y$ of the $i$-th agent. 
		
		We assume $N=100$ nodes working together to solve the problem, and each node has its local dataset $Y_i$. Three different communication networks are considered where the underlying graphs are ring graph, ER graph ($p=0.2$), and complete graph, respectively. Weights of the adjacency matrix are constructed by the Metropolis rule.
		
		We implement our proposed algorithms in the following.
		
		{\noindent\bf DR-SSG.}  We need to  evaluate the local subgradient of each agent by randomly selecting a data sample $\xi_i$ and calculating $\partial_R(f_i(X_i,\xi_i)) =\mathcal{P}_{T_{X_i}\text{St}(n,n)} (y_{(i,\xi_i)} \text{sgn}(y_{(i,\xi_i)}^\top X_i))$, where $\xi_i$ represents the data sample that we selected from $Y_i$.
		
		{\noindent\bf DR-SPP.} At every iteration $k$, each agent needs to solve:
		\begin{equation}\label{drspp}
		\min_{v \in T_{X_{i,k}} \text{St}(n,n)}~\|y_{(i,\xi_i)}^\top (X_{i,k}+v)\|_1 +\tfrac{\beta_k}{2}\|v\|^2.
		\end{equation}
		
		{\noindent\bf DR-SPL.} The distributed proximal linear algorithm also solves:
		$\min_{v \in T_{X_{i,k}} \text{St}(n,n)}~\| y_{(i,\xi_i)}^\top X_{i,k} + y_{(i,\xi_i)}^\top v \|_1+\frac{\beta_k}{2}\|v\|^2.$
		
		We can observe that the subproblems of DR-SPP and DR-SPL turn out to be the same in this problem. To resolve this, we turn the subproblem into a quadratic optimization problem with linear equality constraints. The detailed derivation is provided in Appendix D.
		
		\subsubsection{Synthetic data }
		The orthogonal dictionary matrix $A\in \text{St}(n,n)$ with $n=30$ is randomly generated from a normal distribution and the sparse matrix $S \in \mathbb{R}^{n \times m}$ by the Gaussian distribution with density $0.2$. Then the data matrix can be derived by $Y = AS$, where $m$ represents the size of the data and $m=10000$. The data are divided equally among the $N$ agents and we have $Y = [Y_1, Y_2, \dots, Y_N]$. Let the step size $\alpha = 1$ and $\beta_k=\frac{\sqrt{k+1}}{\beta_0}$. The retraction uses QR decomposition unless otherwise specified and the projection uses SVD. The initial points of every node are the same point randomly generated on the manifold. We measure the error to the target point by $err(\hat X, A) = \sum_{i=1}^n|\max_j|[\hat X_{[:,i]}^\top A]|_j-1|$.

		First,  we show the convergence performance of our algorithms with different $t$. Consider the ER graph. As illustrated in Fig. \ref{fig:t}, the theoretical lower bound for ER graph is approximately $4 \times 10^2$, while the convergence behaviors of DR-SSG, DR-SPP\&DR-SPL actually achieve similar performance when $t=1$, which implies that the theoretical bound of $t$ can be relaxed in practice. Thus, we simply set $t=1$ in the following.
		
		\begin{figure} [htbp]
			\centering
			\setlength{\abovecaptionskip}{0cm}
			\subfloat[DR-SSG, $\beta_0 = 0.1$]{ \includegraphics[width=0.49\columnwidth]{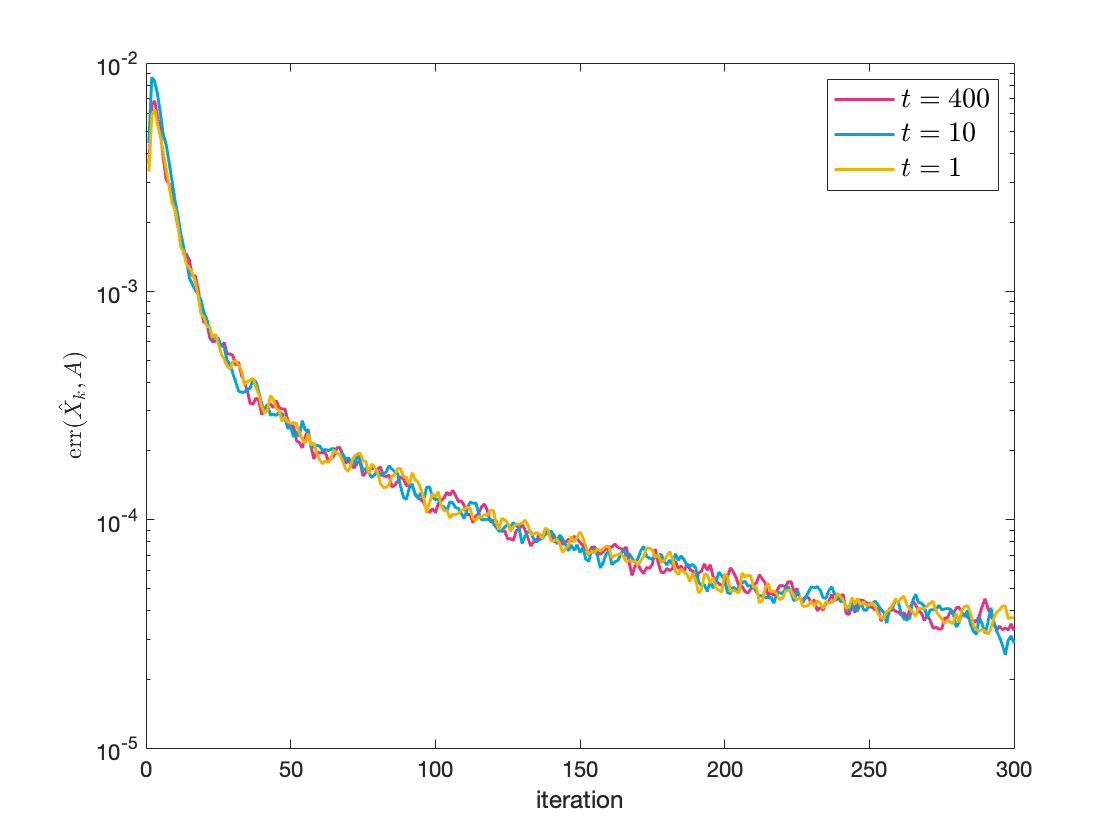} }
			\subfloat[DR-SPP\&DR-SPL, $\beta_0=0.01$]{ \includegraphics[width=0.49\columnwidth]{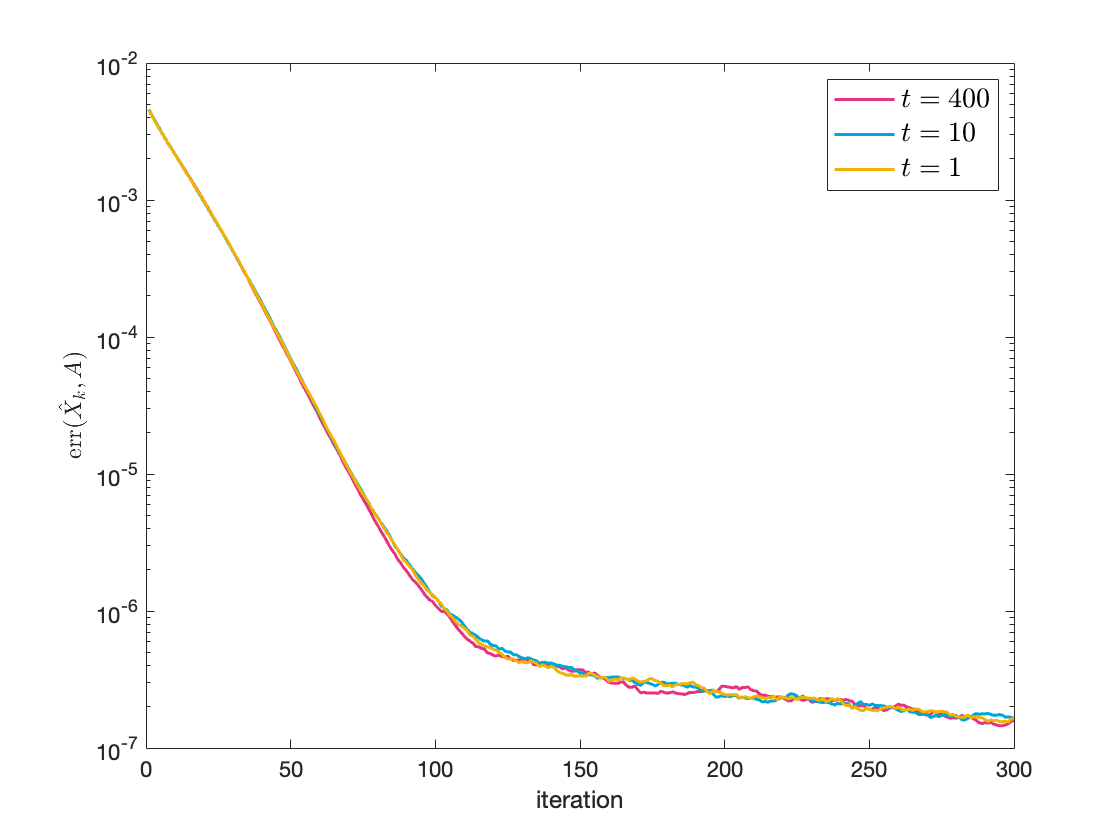} }
			\caption{The convergence performance with $t= 1, 10, 400$.}
			\label{fig:t}
		\end{figure}

		Then, we compare our algorithm with existing works, especially the projection-based algorithms in \cite{deng2023decentralized, wang2024proxtrack}. Note that the proposed DR-SSG is the stochastic version of \cite{wang2023} on $\text{St}(n,n)$. Since \cite{deng2023decentralized} is developed for smooth cost functions, we test a variant called DPSGD, by replacing the full gradient with a stochastic subgradient. We also run the projection counterpart of DR-SPP, labeled as DP-SPP, which is the stochastic version of \cite{wang2024proxtrack} with the smooth term $f(X) = 0$ and the nonsmooth term $r(X)$ known locally to each agent (since the gradient tracking step in DRProxGT tracks the gradient of $f(X)$, it vanishes here). The results are displayed in Fig. \ref{fig-1}. 
			
			\begin{figure*}[htbp]
				\centering
				\setlength{\abovecaptionskip}{0cm}
				\subfloat[Ring Graph]{ \includegraphics[width=0.24\linewidth]{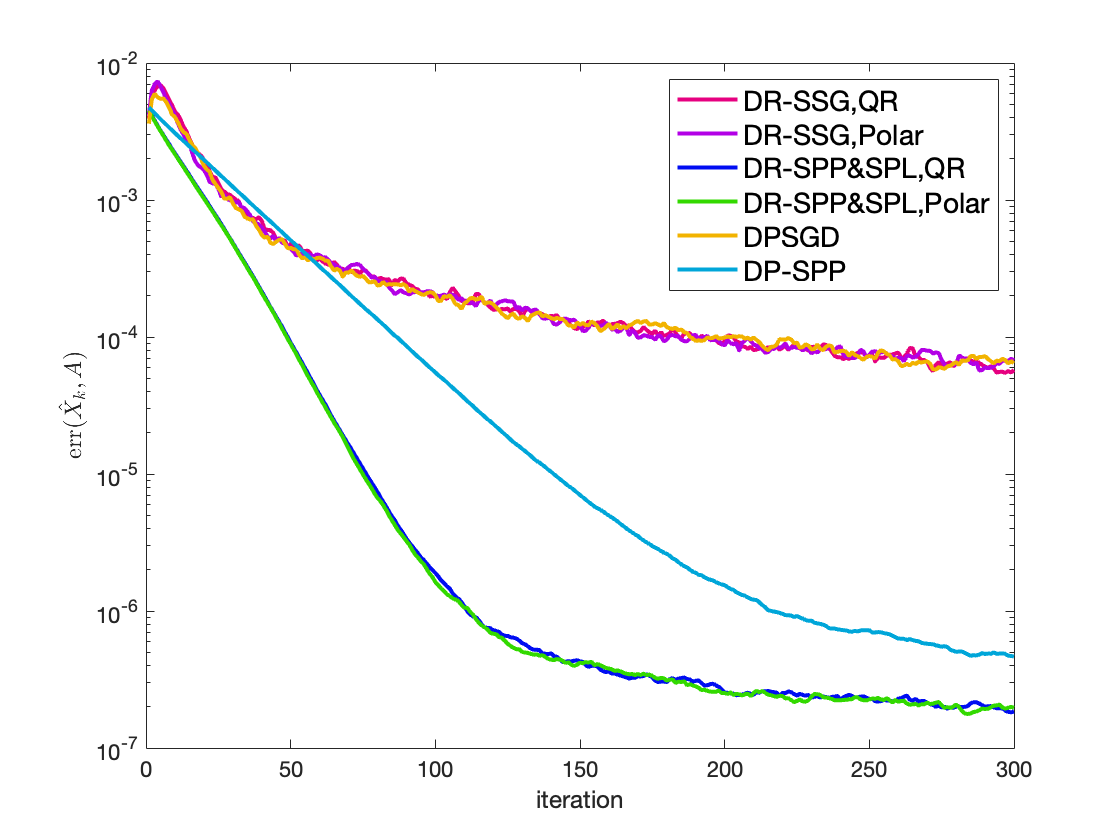} }
				\subfloat[ER Graph($p=0.2$)]{ \includegraphics[width=0.24\linewidth]{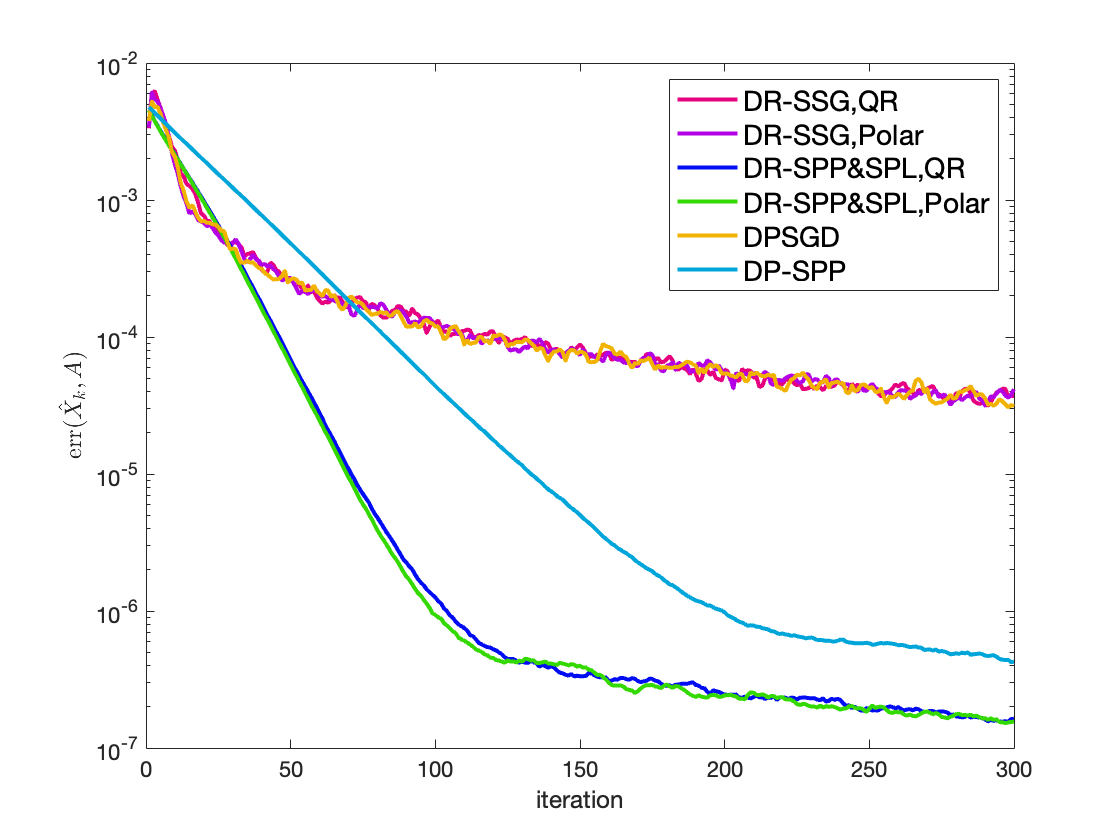} }
				\subfloat[Complete Graph]{ \includegraphics[width = 0.24\linewidth]{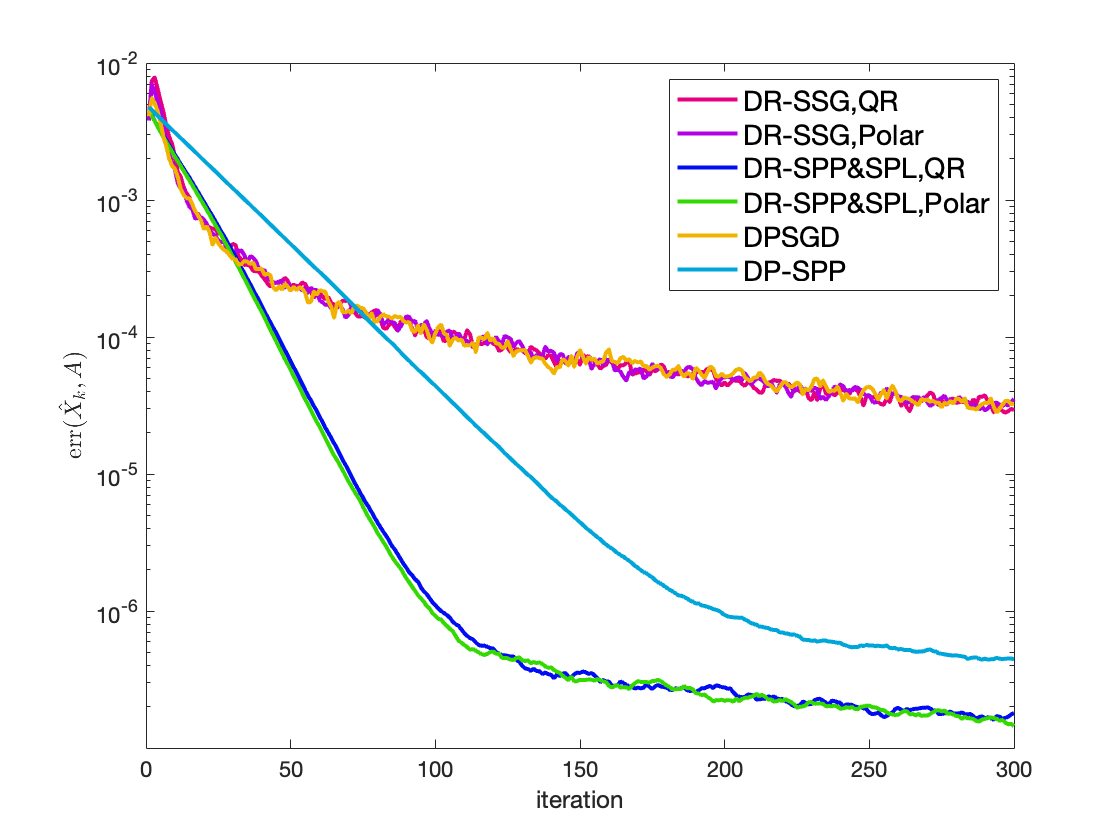} }
				\subfloat[Runtime]{ \includegraphics[width = 0.24\linewidth]{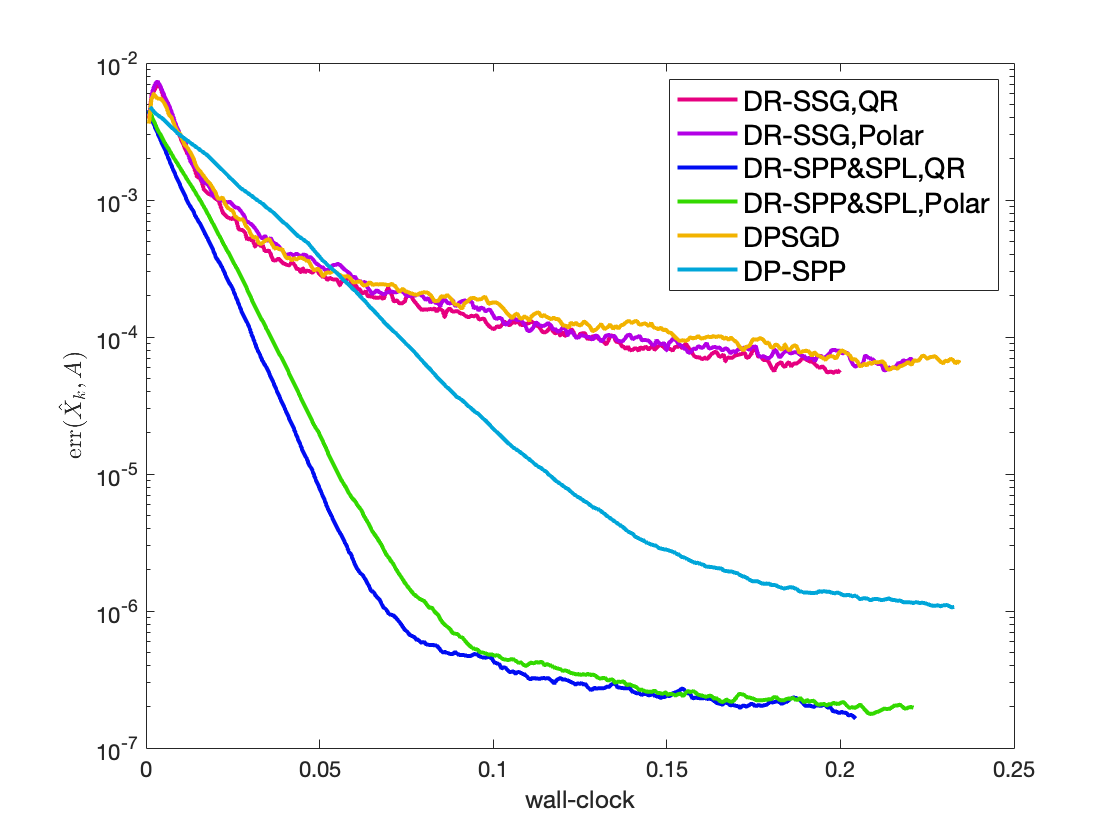} }
				\caption{The performance of the proposed algorithms, DPSGD, and DP-SPP under different communication graphs, where $\beta_0 = 0.1$ for DR-SSG and DPSGD, $\beta_0=0.01,~0.02$ for DR-SPP\&DR-SPL and DP-SPP respectively.}
				\label{fig-1}
			\end{figure*}
			
			Fig. \ref{fig-1}(a)-(c) illustrate that DR-SPP\&DR-SPL exhibits faster convergence rate and achieves lower recovery errors among all the algorithms. In particular, DR-SPP\&DR-SPL and DP-SPP continue to improve throughout the iterations, while DR-SSG and DPSGD plateau at approximately $10^{-4}$. Using retractions leads to a better convergence performance for the proximal-point methods as, while has little impact on subgradient algorithms DR-SSG and DPSGD. 
			
			\begin{table}[ht]
				\centering
				\setlength{\abovecaptionskip}{0cm}
				\caption{Comparison of the average wall-clock per iteration between four algorithms.}
				\begin{tabular}{lccc}
					\toprule
					\multirow{2}{*}{Algorithms}& \multicolumn{3}{c}{wall-clock per iteration (${\times 10^{-4}}$, unit:s)}  \\
					\cline{2-4}
					\rule{0pt}{2.6ex} &ring&ER&complete\\
					\midrule
					
					DR-SSG
					& 6.52 	& 6.82& 6.81\\
					\midrule
					
					DR-SPP\&DR-SPL
					& 6.75 	& 6.88& 6.90\\
					\midrule
					
					DPSGD
					& 7.82 	& 7.55 & 7.65\\
					\midrule
					
					DP-SPP
					& 7.96 	& 7.93 & 7.94\\
					\bottomrule
				\end{tabular}
				\label{tab:wall-clock}
			\end{table}
			
			To assess computational efficiency, Table~\ref{tab:wall-clock} shows the average wall-clock time per iteration and Fig. \ref{fig-1}(d) gives the wall-clock v.s. convergence. We can see that retraction-based methods require less computation per iteration than projection-based ones, demonstrating higher computational efficiency.
			Although DR-SPP\&DR-SPL and DP-SPP require more computation per iteration than DR-SSG and DPSGD as shown in Table~\ref{tab:wall-clock}, they achieve the target accuracy with fewer iterations and therefore provide better overall computational performance as illustrated in Fig.~\ref{fig-1}(d).
			
			We further compare the proposed algorithms using QR and polar retractions in Fig. \ref{fig-1}. As shown in Fig. \ref{fig-1}(a)-(c), the convergence behavior of both DR-SSG and DR-SPP is largely insensitive to the retraction employed. 
			Moreover, Fig.~\ref{fig-1}(d) shows that the QR retraction requires less wall-clock time while achieving comparable convergence accuracy, making it a more computationally efficient choice in practice.
			
			In conclusion, the proposed retraction-based algorithms offer higher efficiency making them particularly attractive for large-scale problems, and may further yield higher accuracy for certain types of algorithms.

		\subsubsection{Real-word data}
			We also test our algorithm on real-world dataset \href{https://www2.eecs.berkeley.edu/Research/Projects/CS/vision/grouping/resources.html?utm_source=chatgpt.com}{BSDS500} from \cite{amfm_pami2011} to see whether the complete dictionaries are reasonable sparsification bases for real images. 
			
			We generate the training dataset by extracting overlapping $8\times 8$ patches from the training images of the BSDS500 dataset. For each image, a fixed number of patches are randomly sampled, mean-centered individually, and concatenated to form the data matrix $Y \in \mathbb{R}^{64 \times m}$, where $m =10000$. The data matrix is normalized to have identical nonvanishing singular values and separated into each node. We set $t=1$, $\beta_k = \sqrt{k+1} \times 10^{-2}$ and the underlying graph to be ring graph.
			The results are shown in Fig. \ref{fig2}. 
			
			\begin{figure} [!h]
				\centering
				\setlength{\abovecaptionskip}{0cm}
				\subfloat{ \includegraphics[width=0.49\columnwidth]{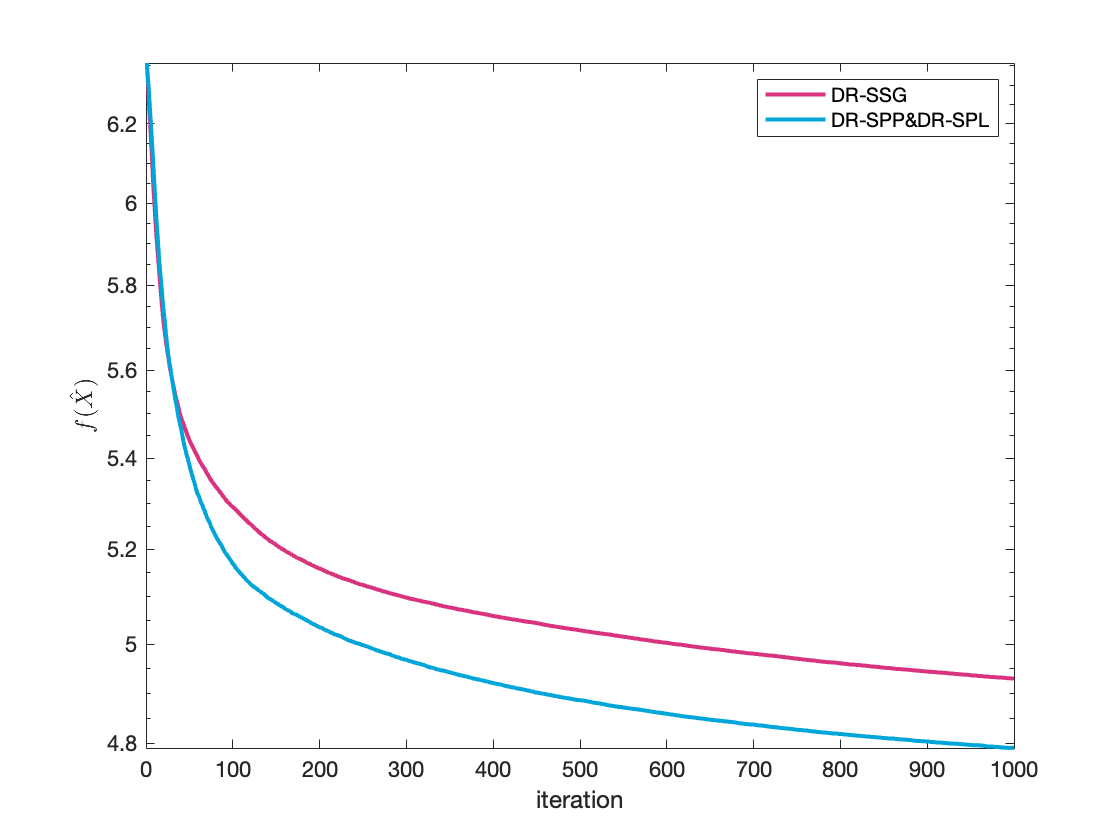} }
				\subfloat{ \includegraphics[width=0.50\columnwidth]{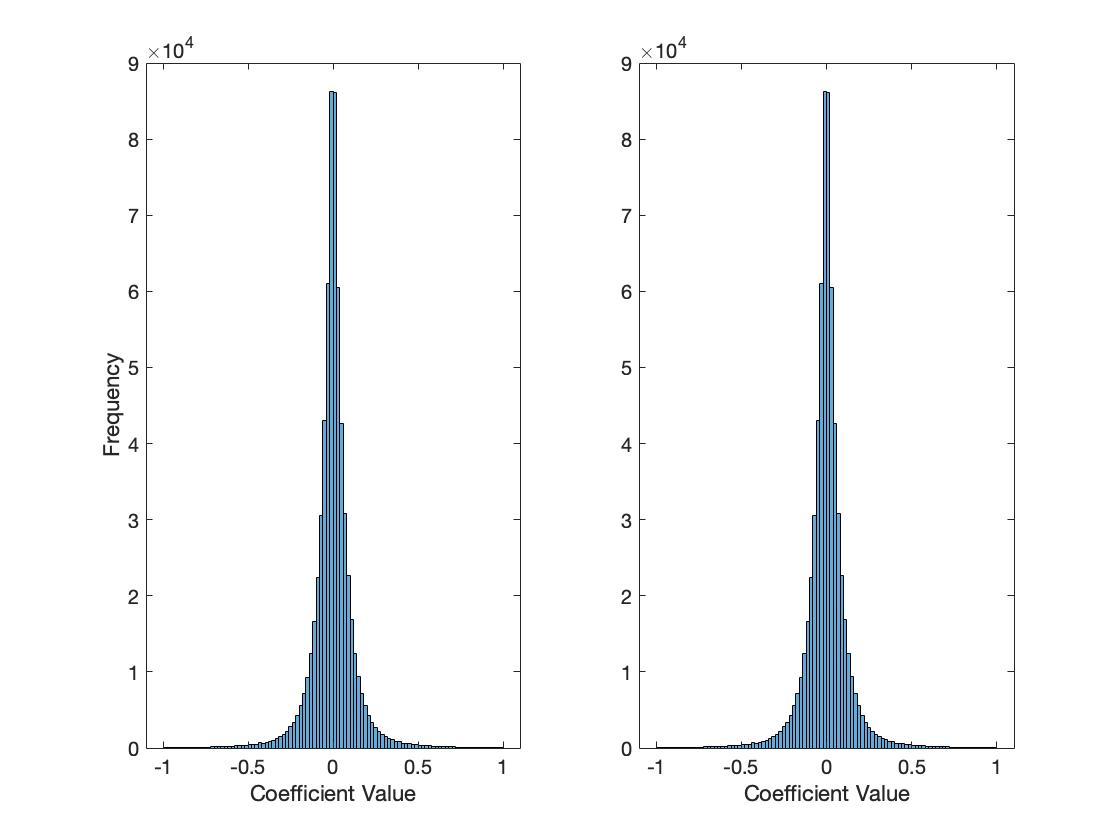} }
				\caption{Results on BSDS500 dataset. Left: Objective values; Right: Histograms of the representation coefficients.}
				\label{fig2}
			\end{figure} 
			
			By the figure of objective values, we can see that the DR-SPP (DR-SPL) algorithm performs better that DR-SSG.  Let $\hat X$ be the learned dictionary and we estimate the representation coefficients as $\hat X^{-1} Y$. The histograms of the representation coefficients of both algorithms are sharply
			concentrated around zero and presents a heavy tailed distribution, which is a good indication of sparsity. We further calculate the mean sparsity level of the coefficient's column vectors by $\frac{\|\cdot\|_1}{\|\cdot\|_2}$ \cite{bai2018subgradient} which ranges from $1$ to $\sqrt{n} = 8$ (fully dense). The sparsity levels of DR-SSG and DR-SPP\&DR-SPL are $4.98$ and $4.85$, respectively. To sum up, the learned dictionaries by our algorithms are reasonable sparsification bases for the dataset.
			
		\subsection{Distributed Generalized Eigenvalue Problem (GEP) on The Generalized Stiefel Manifold}
		We further show  the impact of geodesic curvature parameter $\kappa_g$ through the top-$p$ distributed generalized eigenvalue problem on generalized Stiefel manifold $St_B(n,p):=\{X|X^\top B X = I_p\}$ with positive-definite $B\in \mathbb{R}^{n\times n}$. It can be modeled as 
		$\min_{X\in St_B(n,p)} ~ -\frac{1}{2N}\sum_{i=1}^N tr(X^\top A_i^\top A_i X),$ where $A_i \in \mathbb{R}^{m_i \times n}$ is the local data matrix of agent $i$ and $m_i$ is the data size. Denote by $A:=[A_1^\top, \dots, A_N^\top]^\top$ the global data matrix. 
		According to Remark \ref{r-3}  and the Weingarten map of $St_B(n,p)$ \cite{shustin2023riemannian}, $\kappa_g = \mathcal{O} \left( \frac{\lambda_{\max}(B)^{3/2}}{\lambda_{\min}(B)} \right)$, which is positive related to the condition number of $B$. Since the cost function is smooth, the subgradient used in DR-SSG is actually the stochastic gradient, and DR-SPL algorithm degenerates to DR-SSG with $H_i(x) = x$ and $c_i = \tilde{F}_i$.
		
		We set $N = 10$ nodes and fix $m_1 = \dots = m_{10} = 1000$, $n = 10$ and  $p=5$. Data matrix $A$ is generated following a standard multivariate Gaussian distribution and randomly separated to each node.  The condition numbers of $B$ are set to be $10$ and $100$, respectively. We run both algorithms under ER graph ($p = 0.6$). Let the multi-steps of consensus $t=1$, and the step size $\alpha = 1$. The step size $\beta_k=\frac{\sqrt{k+1}}{0.5}$ for DR-SGG and $\beta_k=\sqrt{k+1}$ for DR-SPP. The initials of every node are the same and generated randomly. The results are illustrated in Fig. \ref{fig3}. It shows that a smaller $\kappa_g$ leads to a better convergence rate, which is consistent with the theoretical result. Besides, small geodesic curvature can make the algorithms more stable.
		\begin{figure} [!h]
			\centering
			\setlength{\abovecaptionskip}{0cm}
			\subfloat[DR-SSG\&DR-SPL]{ \includegraphics[width=0.49\columnwidth]{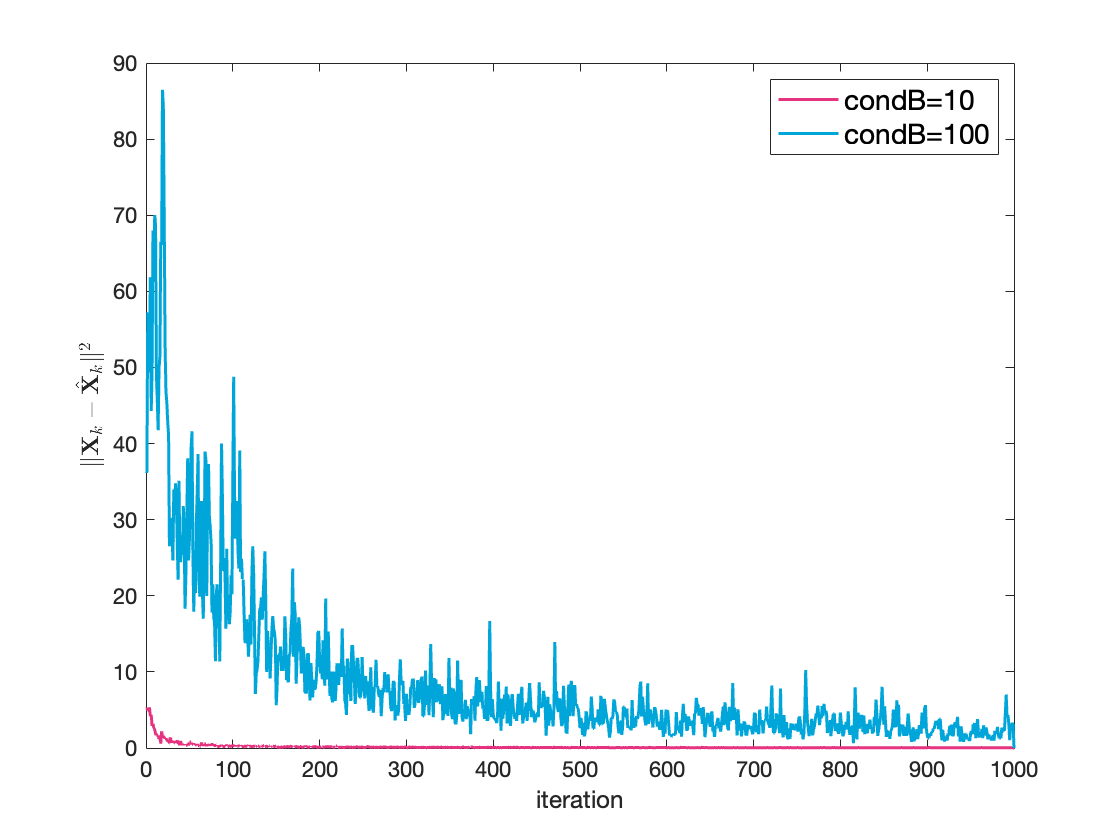} }
			\subfloat[DR-SPP]{ \includegraphics[width=0.49\columnwidth]{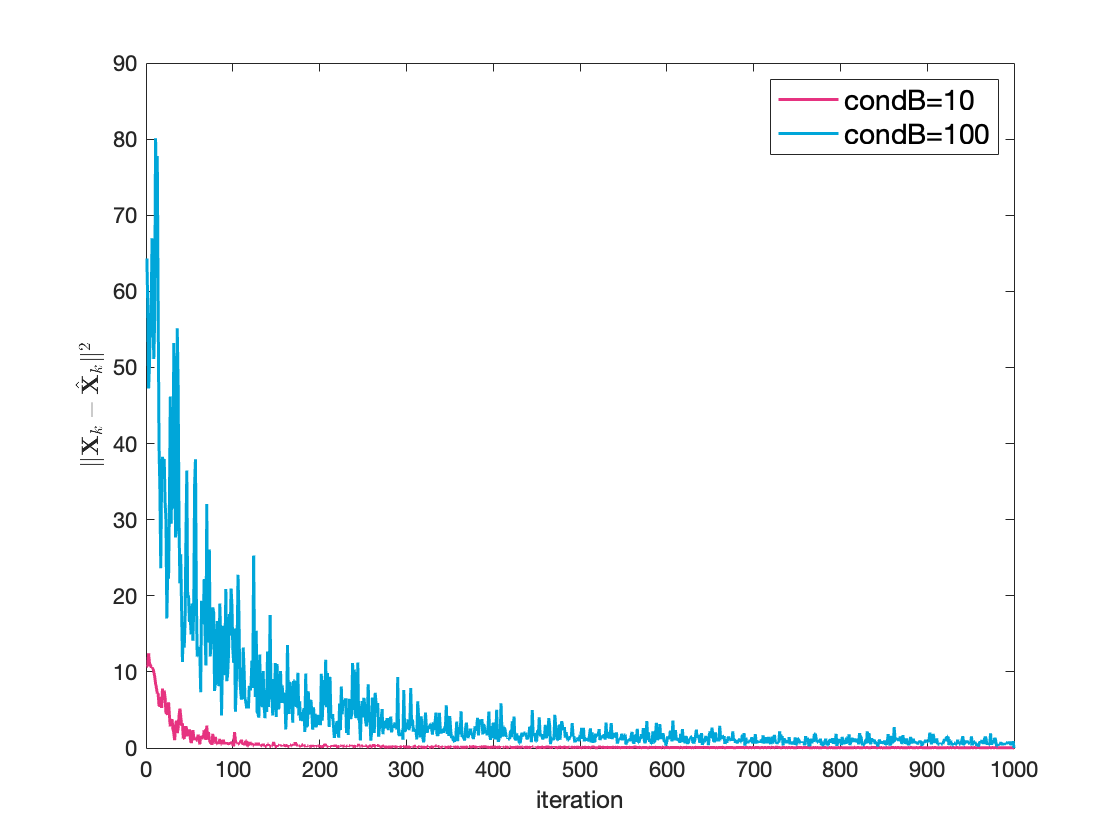} }
			\caption{Distributed generalized eigenvalue problem: The comparison of different $\kappa_g$ on DR-SSG\&DR-SPL and DR-SPP algorithms, respectively.}
			\label{fig3}
		\end{figure} 
		
		\section{Conclusion}\label{sec-6}
		We have proposed a distributed Riemannian stochastic proximal algorithm framework for distributed stochastic optimization problems on compact embedded submanifolds over connected networks. Using retractions, the iterates remain on the manifold and stay within a local region under suitable initialization. Under weakly-convex and Lipschitz continuous conditions on local cost functions and stochastic models, we prove consensus and convergence to a nearly stationary point in expectation with rate $\mathcal{O}((1+\kappa_g)/\sqrt{k})$, explicitly quantifying the influence of geodesic curvature. We believe that integrating gradient tracking into the proposed framework is a nontrivial extension to further accelerate convergence and relaxing the conservative conditions such as multi-step $t$ is also interesting, which we leave for future work.
		
		\section*{Appendix}
		\subsection{Proof of Lemma \ref{lem-lp}:}
		Suppose $x,y \in \bar{U}(d)$, by using \eqref{r_cur}, we have the following inequalities for $x$ and $y$ respectively.
		\[\langle \tfrac{x - \mathcal{P}|_\Omega (x)}{d(x,\mathcal{M})}, \mathcal{P}|_{\Omega}(y) - \mathcal{P}|_{\Omega}(x) \rangle \leq \tfrac{\kappa_g}{2} d_\mathcal{M}(\mathcal{P}|_{\Omega}(x),\mathcal{P}|_{\Omega}(y))^2, \] 
		\[\langle \tfrac{y - \mathcal{P}|_\Omega (y)}{d(y, \mathcal{M})}, \mathcal{P}|_{\Omega}(x) - \mathcal{P}|_{\Omega}(y) \rangle \leq \tfrac{\kappa_g}{2} d_\mathcal{M}(\mathcal{P}|_{\Omega}(x), \mathcal{P}|_{\Omega}(y))^2.\]
		Since $d(x,\mathcal{M}) \leq d$ and $d(y,\mathcal{M})\leq d$, summing the above two inequalities yields \begin{align}\label{k2}
		&\langle x - y + \mathcal{P}|_\Omega (y) -  \mathcal{P}|_\Omega (x), \mathcal{P}|_\Omega (y) -  \mathcal{P}|_\Omega (x) \rangle \notag\\
		\leq& (\kappa_g d) d_\mathcal{M}(\mathcal{P}|_{\Omega}(x),\mathcal{P}|_{\Omega}(y))^2.
		\end{align}
		Reformulating \eqref{k2} implies $
		\langle x - y , \mathcal{P}|_\Omega (y) -  \mathcal{P}|_\Omega (x) \rangle \leq  (\kappa_g d) d_{\mathcal{M}}(\mathcal{P}|_{\Omega}(x), \mathcal{P}|_{\Omega}(y))^2- \|\mathcal{P}|_\Omega (y) -  \mathcal{P}|_\Omega (x)\|^2 
		\leq (\kappa_g d-\frac{1}{2}) d_\mathcal{M}(\mathcal{P}|_{\Omega}(x), \mathcal{P}|_{\Omega}(y))^2,
		$where the last inequality follows by \eqref{egeqr}.
		
		Consequently, $d_\mathcal{M}(\mathcal{P}|_{\Omega}(x)- \mathcal{P}|_{\Omega}(y))^2\leq \frac{2}{1-2\kappa_g d}\|x-y\|d_\mathcal{M}(\mathcal{P}|_{\Omega}(x),\mathcal{P}|_{\Omega}(y)) $. 
		Thus, \eqref{lp} holds.   \hfill$\blacksquare$
		
		\subsection{Proof of Theorem \ref{th-1}:}
		
		We prove $\bold{X}_k \in \mathcal{S},~\forall k\geq 0$ by induction. This is satisfied when $k=0$ according to the initial condition. By the induction hypothesis, there exists a $k_0$, for every $0 < k \leq  k_0$, it holds $\bold{X}_{k} \in \mathcal{S} $.
		
		Then for $k=k_0+1$, we first show that $\bold{X}_{k_0+1} \in \mathcal{S}_2$. 
		Since $\beta_{k_0} \geq \frac{C_1L}{\delta_2\sqrt{1-\rho_t^2}}$, by Lemma \ref{c-e}, we have 
		$\|\bold{X}_{k_0+1} - \bold{\hat X}_{k_0+1}\|^2 \leq \rho_t^2\| \bm{X}_{k_0} - \bold{\hat X}_{k_0}\|^2 +\tfrac{NC_1^2L^2}{\beta_{k_0}^2}  \leq N\delta_2^2.$
		
		Next, we verify $\bold{X}_{k_0+1}\in \mathcal{S}_1$. Noticing that $\|X_{i,k_0+1} - \hat X_{k_0+1}\| \leq \|X_{i,k_0+1} - \hat X_{k_0}\|  + \|\hat X_{k_0} - \hat X_{k_0+1}\|$. For $\forall i \in \mathcal{M}$, utilizing \eqref{R-2} implies 
		$
		\|X_{i,k_0+1} - \hat X_{k_0}\|  \leq \|X_{i,k} -\alpha \text{grad}\ h_{i,t}(\bold{X}_k)+ v_{i,k_0} - \hat X_k\| + 4\kappa_g \alpha \delta_1^2+\tfrac{M_2L^2}{\beta_{k_0}^2}+\tfrac{ 4\kappa_g \delta_1L}{\beta_{k_0}}$.
		Since $t \geq \lceil \log_{\sigma_2}(\frac{1}{5\sqrt{N}})\rceil$, and by using Lemma \ref{ave} and Lemma \ref{nv}, we derive that 
		\begin{align*}
		&\| X_{i,k}-\alpha \text{grad}\ h_{i,t}(\bold{X}_{k_0})+ v_{i,k_0} - \hat{X}_{k_0}\| \\
		\leq& \|(1-\alpha)(X_{i,k_0}-\hat{X}_{k_0})+\alpha (\bar{X}_{k_0}- \hat{X}_{k_0})+\|v_{i,k_0}\|\\
		+&\alpha \sum_{j=1}^N W_{ij}^t(X_{j,k_0}-\bar X_{k_0})+\alpha \mathcal{P}_{N_{X_{i,k_0}}\mathcal{M}}(\nabla h_{i,t}(\bold{X}_{k_0}))\| \\
		\leq &(1-\alpha)\|X_{i,k_0}-\hat{X}_{k_0}\|+\alpha \|\bar{X}_{k_0}- \hat{X}_{k_0}\|+\|v_{i,k_0}\|\\
		+&\alpha \|\sum_{j=1}^N (W_{ij}^t - \frac{1}{N})X_{j,k_0}\|+\alpha \|\mathcal{P}_{N_{X_{i,k_0}}\mathcal{M}}(\nabla h_{i,t}(\bold{X}_{k_0}))\|\\
		\leq& (1-\frac{4}{5}\alpha)\delta_1 +\kappa_g \alpha \delta_2^2 +4\kappa_g \alpha\delta_1^2 + \frac{L}{\beta_{k_0}}.
		\end{align*}
		According to \eqref{R-2} and \eqref{e-iter}, it follows that
		\begin{align*}
		\|\bar X_{k_0} - \bar X_{k_0+1}\| \leq & \frac{M_2}{N}\sum_{i=1}^{N}\|\alpha\text{grad}\ h_{i,t}(\bold{X}_{k_0})-v_{i,k_0}\|^2 \\
		+& \alpha \|\frac{1}{N} \sum_{i=1}^{N}\text{grad}\ h_{i,t}(\bold{X}_{k_0})  \|+ \frac{1}{N}\sum_{i=1}^{N}\|v_{i,k_0}\|\\
		\leq& 2\alpha L_t(L_t\alpha M_2+\kappa_g) \delta_2^2+\tfrac{2M_2L^2}{\beta_{k_0}^2}+ \tfrac{L}{\beta_{k_0}}.
		\end{align*}
		Since $\alpha \leq \frac{\kappa_g}{M_2}$, $L_t \leq 2$ and $\beta_k \geq \frac{5L}{\alpha\delta_2}$, we finally have
		$\|X_{i,k_0+1} - \hat X_{k_0+1}\| \leq \|X_{i,k_0+1} - \hat X_{k_0}\|  + \|\hat X_{k_0} - \hat X_{k_0+1}\| \leq \delta_1.$
		
		Lastly, we show that $\bold{X}_{k_0+1} \in \mathcal{S}_3$. Since we have proved that $\bold{X}_{k_0+1}\in \mathcal{S}_1$, together with Corollary \ref{e-r ave} and $\alpha \leq 1$, the Riemannian distance between $\hat{X}_{k_0}$ and $\hat{X}_{k_0+1}$  satisfies
		\begin{equation}\label{s-6}
		d_{\mathcal{M}}(\hat{X}_{k_0}, \hat{X}_{k_0+1}) \leq 3\|\bar X_{k_0} - \bar X_{k_0+1}\| \leq \tfrac{906 \kappa_g \delta_2^2}{25}+\tfrac{3 \delta_2}{5}.
		\end{equation}
		By (\romannumeral3) of Lemma \ref{g} and Lemma \ref{b-v}
		\begin{equation}\label{s-7}
		d_{\mathcal{M}}(X_{i,k_0+1}, {X}_{i,k_0})\leq \|\mathcal{R}_{X_{i,k_0}}^{-1}(X_{i,k_0+1})\|\leq 2\delta_1 + \tfrac{ \delta_2}{5}
		\end{equation}
		By combining \eqref{s-6} and \eqref{s-7}, and together with $\bold{X}_{k_0}\in \mathcal{S}_3$, there holds \begin{align*}
		&d_{\mathcal{M}}(X_{i,k_0+1}, \hat{X}_{k_0+1}) = d_{\mathcal{M}}(X_{i,k_0+1}, {X}_{i,k_0})\\
		+&d_{\mathcal{M}}(X_{i,k_0}, \hat{X}_{k_0}) + d_{\mathcal{M}}(\hat{X}_{i,k_0}, \hat{X}_{k_0+1}) \leq \frac{1}{\kappa_g}
		\end{align*}
		According to \eqref{dist} and rearranging the terms, we have
		\begin{align*}
		&\|\text{Exp}^{-1}_{X_{i,k_0+1}}(\hat{X}_{k_0+1}) \|-\tfrac{\kappa_g}{2} \|\text{Exp}^{-1}_{X_{i,k_0+1}}(\hat{X}_{k_0+1})\|^2\\
		\leq& \| \hat{X}_{k_0+1}-X_{k_0+1}\|  \leq \delta_1,
		\end{align*}  which implies that 
		$d_{\mathcal{M}}(X_{i,k_0+1}, \hat{X}_{k_0+1}) = \|\text{Exp}^{-1}_{X_{i,k_0+1}}(\hat{X}_{k_0+1}) \| \leq \frac{1-\sqrt{1-2\kappa_g \delta_1}}{\kappa_g}\leq \frac{2-\sqrt{2}}{\kappa_g}$. \hfill$\blacksquare$
		
		\subsection{Proof of Theorem \ref{th-c}:}
		
		(\romannumeral1) is followed by using telescopic cancellation on \eqref{c-err}. 
		
		Then, we prove (\romannumeral2) by induction. Since $\lim_{k \rightarrow\infty} \frac{\beta_{k+1}}{\beta_k}=1$ according to Assumption \ref{beta}, there exists a sufficiently large $G\geq 0$ such that for $\forall k\geq G$, $\frac{\beta_{k+1}}{\beta_k} \leq 2.$ Denote the sequence $z_k:=\frac{\beta_k\| \bm{X}_k - \bold{\hat X}_{k}\|}{\sqrt{N}}$, $k\geq 0$. For $\forall k\geq 0$, we intend to show that $z_k \leq \mathcal{O}(L^2)$.
		
		When $k=0$, we have $\beta_0=\mathcal{O}(L)$ and $\frac{1}{N}\| \bm{X}_0 - \bold{\hat X}_{0}\|^2 \leq \delta_2^2$ which imply that $z_0^2 = \mathcal{O}(L^2)$. By the induction hypothesis, for $0\leq k \leq G$, there exists a constant $\tilde{C}$ ($\tilde{C}$ is independent of $L$), such that $z_k^2 \leq  \tilde{C} L^2.$
		
		For $k \geq G+1$, it follows from \eqref{c-err} that
		$z_{k+1}^2\leq \rho_t^2z_k^2+(C_1L)^2\left(\tfrac{\beta_{k+1}}{\beta_k}\right)^2 \leq (\rho_t^{2(k+1-G)} \tilde{C}+\tfrac{4C_1^2}{1-\rho_t^2})L^2
		\leq (\tilde{C}+\tfrac{4C_1^2}{1-\rho_t^2})L^2.$
		Let $C = (\tilde{C}+\frac{4C_1^2}{1-\rho_t^2}) = \mathcal{O}(\frac{1}{1-\rho_t^2})$, then we finish the proof.\hfill$\blacksquare$
		
		\subsection{Solution of the subproblems}
			{\bf Orthogonal Sparse Dictionary Learning:}
			Since $\text{St}(n,n)$ is an orthogonal group, by using the tangent-space parameterization $v=XU,~ U^\top=-U$, \eqref{drspp} can be rewritten as
			$
			\min_{U^\top=-U}~\|a+U^\top a\|_1+\tfrac{\beta}{2}\|U\|_F^2,
			$ where $a:=X_{i,k}^\top y_{(i,\xi_i)}$. Denote $u=U^\top a$. 
			The subproblem can be further reduced to 
			$min_{a^\top u = 0}\|a+u\|_1+\frac{\beta_k}{\|a\|^2}\|u\|^2,$
			whose dual problem can be solved by computing a monotone equation based on standard numerical solvers (eg. bisection, Newton method).
			
			For the Stiefel manifold $\text{St}(n,r)$ where $n \neq  r$, the subproblem can be solved following the adaptive Semi-smooth Newton method in \cite{xiao2018regularized, chen2020proximal}.

			{\bf GEP on Generalized Stiefel manifold:}
			Let $A_{\xi_i}$ denote the stochastic data sample and $M = A_{\xi_i}^\top A_{\xi_i} \succeq 0$.  The subproblem for DR-SPP is $min_{v} ~ -\frac{1}{2} tr((X_{i,k}^\top+v) M (X_{i,k}+v)) + \frac{\beta_k}{2}\|v\|^2$, s.t. $v^\top B X_{i,k}+X_{i,k}^\top B v=0.$ Vectoring $v$, we have
			$
			\min_{\mathbf v}~
			\tfrac{1}{2}\mathbf v^\top
			\left(\beta_k I-M\otimes I\right)\mathbf v
			-\operatorname{vec}\!\left(X_{i,k}M\right)^\top \mathbf v,~
			\text{s.t.}~
			(I+K)\left(I\otimes X_{i,k}^\top B\right)\mathbf v=0.
			$
			By calculating its KKT conditions, we obtain an equation and solve it by Newton method (or by projected conjugate gradient method for large-scale cases).

			{\bf Blind Deconvolution on Sphere:}
			The original subproblem of DR-SPP is $\min_{v\in T_{x_{i,k}} S^{n-1},y}~ |\langle a_i,(x_{i,k}+v) \rangle \langle c_i,y \rangle - b_i| +\frac{\beta_k}{2}\|v\|^2+ \frac{\beta_k}{2}\|y-y_{i,k}\|^2;$ and DR-SPL is 
			$\min_{(v_1,v_2) \in T_{x_{i,k}}S^{n-1} \times \mathcal{R}^m }~|\langle a_i,x_{i,k} \rangle \langle c_i,y_{i,k} \rangle + \langle c_i,y_{i,k} \rangle \langle a_i,v_1 \rangle+ \langle a_i,x_{i,k} \rangle \langle c_i,v_2 \rangle-b |+\frac{\beta_k}{2} \|v_1\|^2+\frac{\beta_k}{2} \|v_2\|^2.$
			
			The tangent space of sphere is $x^\top v = 0$. 
			Let $U$ denote the orthogonal basis of $T_{x}S^{n-1}$ that satisfies $U^\top x = 0,~U^\top U = I$, we have $v = Uz,~z\in \mathbb{R}^{n-1}$ which turns the subproblems into unconstrained optimization. DR-SPP is solved by alternating minimization where each block update admits a closed-form solution; and DR-SPL admits a closed-form evaluation.
		
		\section*{References}
		\bibliographystyle{plain}
		\bibliography{reference}
		
		\begin{IEEEbiography}[{\includegraphics[width=1in,height=1.25in,clip,keepaspectratio]{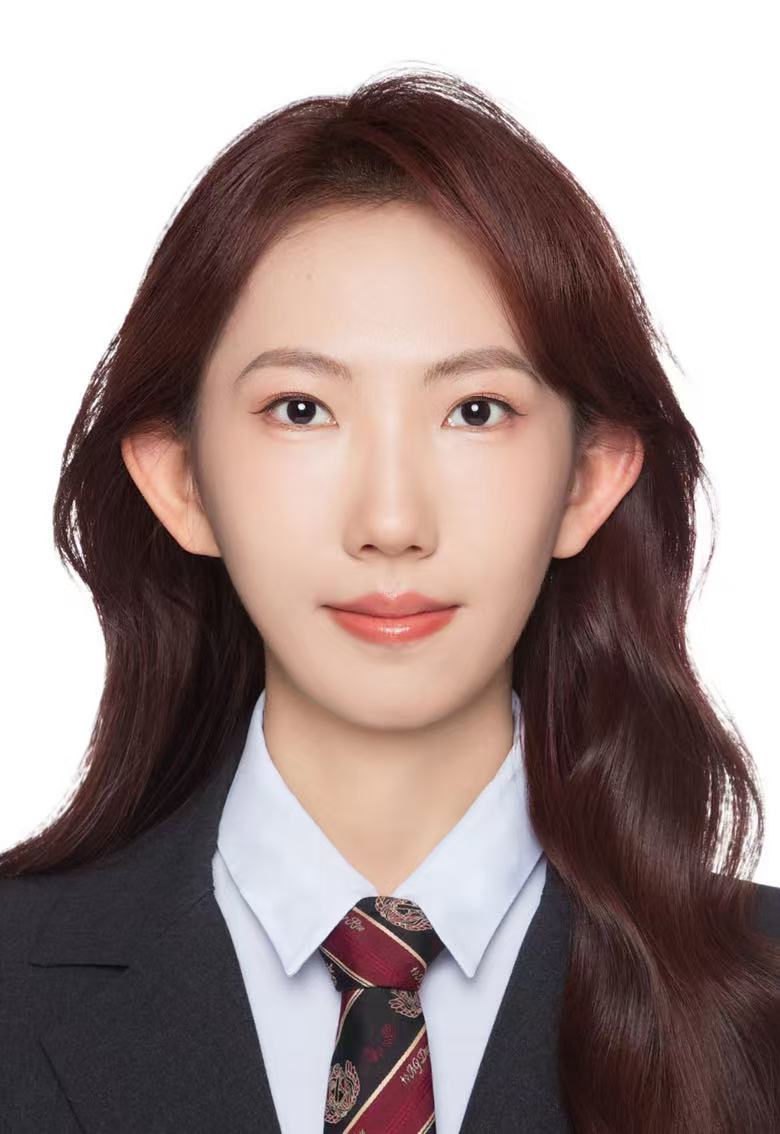}}]{Jishu Zhao} received the B.S. degree from Sichuan University in Mathematics, Sichuan, China, in 2021. She is currently working toward the Ph.D degree in control theory and control engineering with the Tongji University, Shanghai, China. Her research interests include game theory and distributed optimization on Riemannian manifolds.
			%
			%
		\end{IEEEbiography}
		
		\begin{IEEEbiography}[{\includegraphics[width=1in,height=1.25in,clip,keepaspectratio]{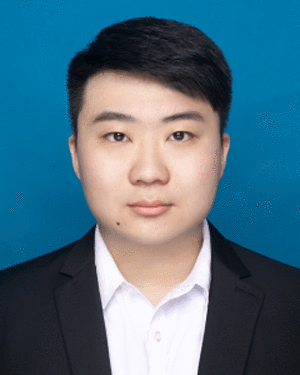}}]{Xi Wang} received the B.S. degree in mathematics and applied mathematics and the Ph.D. degree in operational research and cybernetics from the University of Chinese Academy of Sciences, Beijing, China, in 2018 and 2023, respectively.
			
			He is currently a Postdoctoral Researcher with the School of Electrical Engineering and Telecommunications, The University of New South Wales, Sydney, NSW, Australia. His research interests include online optimization, Riemannian optimization, and Lie group control.
		\end{IEEEbiography}
		
		\begin{IEEEbiography}[{\includegraphics[width=1in,height=1.25in,clip,keepaspectratio]{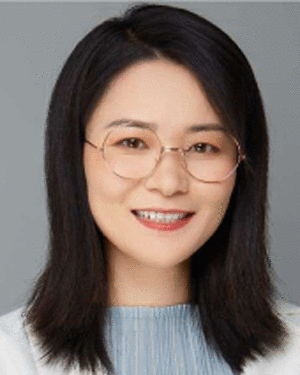}}]{Jinlong Lei} (Member, IEEE) received the B.E. degree in automation from the University of Science and Technology of China, Hefei, China, in 2011, and the Ph.D. degree in operations research and cybernetics from the Institute of Systems Science, Academy of Mathematics and Systems Science, Chinese Academy of Sciences, Beijing, China, in 2016. 
			
			She is currently a Professor with the Department of Control Science and Engineering, Tongji University, Shanghai, China. She was a Postdoctoral Fellow with the Department of Industrial and Manufacturing Engineering, Pennsylvania State University, University Park, PA, USA, from 2016 to 2019. Her research interests include stochastic approximation, stochastic and distributed optimization, and network games under uncertainty.
		\end{IEEEbiography}
		
		\begin{IEEEbiography}[{\includegraphics[width=1in,height=1.25in,clip,keepaspectratio]{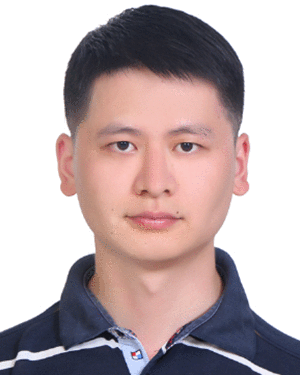}}]{Shixiang Chen} received the PhD degree in systems engineering and engineering management from The Chinese University of Hong Kong, in 2019. 
			
			He is an assistant professor with the School of Mathematical Sciences, University of Science and Technology of China. He was a postdoctoral associate with the Department of Industrial \& Systems Engineering, Texas A\&M University. His current research interests include design and analysis of optimization algorithms, and their applications in machine learning and signal processing.
		\end{IEEEbiography}

\newpage

\section{Supplementary: Additional Proof of Lemmas}
\subsection{Proof of Lemma \ref{ave}:}

For each $i$, there exists a unique geodesic $\gamma_i:[0,1] \mapsto \mathcal{M}$ connecting $X_i$ and $\hat{X}$. We denote the tangent vector along each $\gamma_i(t)$ as $\eta_i(t)$, then it implies
\begin{equation}\label{a-3}
X_i - \hat{X} =\int_{0}^{1}\eta_i(t) dt,~\forall i \in \mathcal{N}.
\end{equation}
Utilizing this result yields $\|\hat{X} - \bar{X}\|^2 = \langle \hat{X} - \bar{X},\hat{X} - \bar{X} \rangle 
=  \frac{1}{N} \sum_{i=1}^{N}\langle \hat{X} - \bar{X},\int_{0}^{1}\eta_i(t) dt \rangle.$

By the first-order optimal condition of \eqref{iam}, we have $\hat{X} - \bar{X}\in \mathcal{N}_{\hat{X}}\mathcal{M}$, where $\mathcal{N}_{\hat{X}}\mathcal{M}$ is the normal cone at $\hat X$.  Thus, for $\forall \eta_i \in T_{\hat{X}}\mathcal{M}$, there holds 
\begin{equation}\label{nc}
\langle \eta_i, \hat{X} - \bar{X}\rangle \leq 0.
\end{equation}
Since $\eta_i(0) \in T_{\hat{X}}\mathcal{M}$, together with \eqref{nc}, we derive that $\|\hat{X} - \bar{X}\|^2 \leq  \frac{1}{N} \sum_{i=1}^{N}\langle \hat{X} - \bar{X},\int_{0}^{1}\eta_i(t) - \eta_i(0) dt \rangle,$
which further implies 
\begin{align}\label{a-2}
\|\hat{X} - \bar{X}\| \leq &\frac{1}{N} \sum\nolimits_{i=1}^{N} \int_{0}^{1}\left \|\eta_i(t) - \eta_i(0)\right\| dt \notag\\
\leq& \frac{1}{N} \sum\nolimits_{i=1}^{N} \int_{0}^{1} \int_{0}^{t} \left \| \nabla_{\eta_i}\eta_i(s)\right\| dsdt\notag\\
\leq & \frac{1}{N} \sum\nolimits_{i=1}^{N} \int_{0}^{1} \int_{0}^{t} \|\Pi(\eta_i(s), \eta_i(s)\|ds dt \\
\leq& \frac{\kappa_g}{2N}\sum\nolimits_{i=1}^{N}  \|\eta_i(0)\|^2.\notag
\end{align}

Following a similar statement in the proof of Lemma \ref{nv} (let $X = \hat X,~Y = X_i$), we have
$\|X_i - \hat{X} - \eta_i(0) \| \leq \frac{\kappa_g}{2}\|\eta_i(0)\|^2$. Since $\|\eta_i(0)\| = d_{\mathcal{M}} (X_i,\hat{X}) \leq D$, we further have $\|X_i - \hat{X}\|\geq \|\eta_i(0)\| (1-\frac{\kappa_g}{2}\|\eta_i(0)\|)\geq \frac{1}{\sqrt{2}}\|\eta_i(0)\|$.
By plugging this result into \eqref{a-2}, we prove the lemma. \hfill$\blacksquare$

\subsection{Proof of Lemma \ref{g}:}

We first prove (\romannumeral1). As the orthogonal projection is self-adjoint and idempotent, it holds that
\begin{align*}
&\|\mathcal{P}_{N_{X_i}\mathcal{M}}(\nabla h_{i,t}(\bold{X}))\|^2 \\
=&  \langle \mathcal{P}_{N_{X_i}\mathcal{M}}(\nabla h_{i,t}(\bold{X})), \mathcal{P}_{N_{X_i}\mathcal{M}}(\nabla h_{i,t}(\bold{X})) \rangle \\
= & \langle \mathcal{P}_{N_{X_i}\mathcal{M}}(\nabla h_{i,t}(\bold{X})), \nabla h_{i,t}(\bold{X}) \rangle,
\end{align*}
and using \eqref{cur} implies \begin{align*}
&\|\mathcal{P}_{N_{X_i}\mathcal{M}}(\nabla h_{i,t}(\bold{X}))\|^2 \\
= &\sum_{j=1}^{N} W_{ij}^t\langle -\mathcal{P}_{N_{X_i}\mathcal{M}}(\nabla h_{i,t}(\bold{X})), X_j - X_i \rangle\\
\leq& \kappa_g\|\mathcal{P}_{N_{X_i}\mathcal{M}}(\nabla h_{i,t}(\bold{X}))\|\sum_{j=1}^{N} W_{ij}^t\|X_i - X_j\|^2.
\end{align*}
The above result shows \begin{equation}\label{nor-h}
\|\mathcal{P}_{N_{X_i}\mathcal{M}}(\nabla h_{i,t}(\bold{X}))\| \leq \kappa_g \sum_{j=1}^{N} W_{ij}^t\|X_i - X_j\|^2.
\end{equation}

Since $\sum_{i=1}^{N} \nabla h_{i,t}(\bold{X}) = \sum_{i=1}^N(X_i - \sum_{j=1}^{N}W_{ij}^t X_j) = 0$, we have 
\begin{align}\label{p_n}
\|\sum_{i=1}^N \text{grad} \ h_{i,t}(\bold{X})\| \leq & \sum_{i=1}^{N}\|\mathcal{P}_{N_{X_i}\mathcal{M}}(\nabla h_{i,t}(\bold{X}))\| \notag\\
\leq &  \kappa_g \sum_{i=1}^{N}\sum_{j=1}^{N} W_{ij}^t\|X_i - X_j\|^2\\
 = &4\kappa_g h_t(\bold{X}).\notag	
\end{align}

For any $\bold{x}, \bold{y} \in \mathcal{M}^N$, there holds\cite{nesterov2013}
\begin{equation}\label{e_l}
h_t(\bold{y}) - [h_t(\bold{x})+\langle \nabla h_t(\bold{x}),\bold{y} - \bold{x}\rangle] \leq \frac{L_t}{2}\|\bold{y} - \bold{x}\|^2,
\end{equation} where $L_t$ is the largest eigenvalue of $\nabla^2 h_t(\bold{x}) = (I_N - W^t) \otimes I_n$ in Euclidean space and $L_t = 1-\lambda_N(W^t)$, where $\lambda_N(W^t)$ denotes the smallest eigenvalue of $W^t$.
Let $\bold{x}=\hat{\bold{X}},~\bold{y}=\bold{X}$ in \eqref{e_l} and note that $\nabla h_t(\hat{\bold{X}}) = 0$. We get \begin{equation}\label{g_x}
2h_t(\bold{X}) \leq L_t\|\bold{X} - \hat{\bold{X}}\|^2,
\end{equation}

By combing \eqref{p_n} and \eqref{g_x}, we prove (\romannumeral1).

Next, we consider to prove (\romannumeral2). In the Euclidean space $\mathbb{R}^n$, based on \eqref{e_l}, we have  \begin{align*}
0 \leq &h_t(\bold{X} - \frac{1}{L_t} \nabla h_t(\bold{X})) \\
\leq & h_t(\bold{X})+\langle \nabla h_t(\bold{X}), -\frac{1}{L_t}\nabla h_t(\bold{X}) \rangle
+\frac{1}{2L_t}\|\nabla h_t(\bold{X})\|^2 \\
= &h_t(\bold{X}) - \frac{1}{2L_t}\|\nabla h_t(\bold{X})\|^2,
\end{align*} which directly shows that\begin{equation}
h_t(\bold{X}) \geq \frac{1}{2L_t}\|\nabla h_t(\bold{X})\|^2.
\end{equation}Since $\text{grad}\ h_{i,t}(\bold{X}) = \mathcal{P}_{T_{X_i}M} (\nabla h_{i,t}(\bold{X}))$, we derive 
\begin{equation}\label{h_t}
\|\text{grad}\ h_t(\bm{X})\|^2 \leq \|\nabla h_t(\bm{X})\|^2 \leq 2L_t h_t(\bold{X})
\end{equation}

By combing \eqref{g_x} and \eqref{h_t}, it follows that (\romannumeral2).

Finally, for (\romannumeral3), since we also have $\bold{X} \in \mathcal{S}_1$, then it holds $\|\text{grad}\ h_{i,t}(\bm{X})\| \leq \|\sum_{j=1}^N W^t_{ij} (X_i - X_j)\| \leq 2\delta_1. $
\hfill$\blacksquare$

\subsection{Proof of Lemma \ref{rs}.} 
Based on the self-adjoint property of the orthogonal projecting operator, we get
\begin{align}\label{c-1}
&\langle  \text{grad}\ h_t(\bm{X}),  \bold{\hat X} - \bm{X} \rangle \notag\\
=& \langle  \nabla  h_t(\bm{X}), \mathcal{P}_{T_{\bold{X}}\mathcal{M}^N} (\bold{\hat X} - \bm{X}) \rangle \notag\\
=& \langle  \nabla  h_t(\bm{X}),  \bold{\hat X} - \bm{X} \rangle - \langle \nabla h_t(\bm{X}),  \mathcal{P}_{N_{\bold{X}}\mathcal{M}^N}(\bold{\hat X} - \bm{X}) \rangle .
\end{align}
It follows from $\nabla h_{i,t}(\bold{X}) = X_i-\sum_{j=1}^{N}W_{ij}^t X_j$ that 
\begin{align}\label{c-2}
&- \langle \nabla h_t(\bm{X}),  \mathcal{P}_{N_{\bold{X}}\mathcal{M}^N}(\bold{\hat X} - \bm{X}) \rangle \notag\\
= &\sum_{i=1}^N \sum_{j=1}^{N}W_{ij}^t \langle X_j- X_i,  \mathcal{P}_{N_{X_i}\mathcal{M}}(\hat X - X_i) \rangle \notag\\
\leq & \kappa_g \sum_{i=1}^N \sum_{j=1}^{N}W_{ij}^t \|\mathcal{P}_{N_{X_i}\mathcal{M}}(\hat X - X_i)\|\|X_i- X_j\|^2\\
\leq & 4\kappa_g \max_{i}\|X_i - \hat X\| h_t(\bold{X}).\notag
\end{align}where the first inequality holds according to \eqref{cur}. 

We then focus on bound the term $\langle \nabla h_t(\bm{X}_k), \bm{X}_k - \bold{\hat X}_{k} \rangle$. For any $\bold{X}\in \mathcal{M}^N$, the potential function of consensus can be rewritten as  $ 2h_t(\bm{X}) = \sum_{i=1}^{N}\|X_{i}\|^2 - \sum_{i=1}^{N}\sum_{j=1}^{N}W^t_{ij}\langle X_{i},X_j \rangle = \langle \nabla h_t(\bm{X}), \bm{X} \rangle.$ Since we also have $\langle \nabla h_t(\bm{X}),  \bold{\hat X} \rangle = 0$, there holds
\begin{equation}\label{c-3}
\langle \nabla h_t(\bm{X}), \bm{X} - \bold{\hat X} \rangle = 2h_t(\bm{X}).
\end{equation}
Substituting \eqref{c-2} and \eqref{c-3} into \eqref{c-1}, we have \begin{align}\label{c-4}
&\langle  \text{grad}\ h_t(\bm{X}),  \bm{X} - \bold{\hat X} \rangle \notag\\
\geq& 2h_t(\bm{X}) - 4\kappa_g \max_{i}\|X_i - \hat X\| h_t(\bold{X}) \\
=& h_t(\bm{X})\cdot (2-4\kappa_g \max_{i}\|X_i - \hat X\|).\notag
\end{align} 

Let $\Phi_{\kappa_g}: = 2-4\kappa_g\max_{i}\|X_i - \hat X\|$. By plugging \eqref{h_t} into \eqref{c-4}, we get $ \langle  \text{grad}\ h_t(\bm{X}),  \bm{X} - \bold{\hat X} \rangle \geq \frac{\Phi_{\kappa_g}}{2L_t}\|\nabla h_t(\bold{X})\|^2 \geq \frac{\Phi_{\kappa_g}}{2L_t} \|\text{grad}\ h_t(\bold{X})\|^2.$ Thus, the lemma is proved.
\hfill$\blacksquare$

\subsection{Proof of Lemma \ref{R-sub}.} 
Since $\phi $ is weakly-convex, and by \eqref{subd}, we obtain \begin{align}\label{l-1}
\phi(Y) - \phi(X) \geq& \langle \tilde{\nabla}\phi(X), Y-X \rangle - \frac{\rho}{2}\|Y-X\|^2 \\
=& \langle g_X, Y-X \rangle + \langle \mathcal{P}_{N_X\mathcal{M}} (\tilde{\nabla}\phi(X)), Y-X \rangle\notag\\
-& \frac{\rho}{2}\|Y-X\|^2.\notag
\end{align} 
Utilizing \eqref{cur} yields \begin{align*}
&\langle \mathcal{P}_{N_X\mathcal{M}} (\tilde{\nabla}\phi(X)), Y-X \rangle\\
 \geq& -\kappa_g\|\mathcal{P}_{N_X\mathcal{M}} (\tilde{\nabla}\phi(X))\|\|Y-X\|^2\\
\geq  &-\kappa_g\|\tilde{\nabla}\phi(X)\|\|Y-X\|^2.
\end{align*} 
According to \cite[Th. 9.13]{2009variational} and since $\phi$ is assumed to be $L$-Lipschitz continuous, we have 
\begin{align}\label{l-2}
\langle \mathcal{P}_{N_X\mathcal{M}} (\tilde{\nabla}\phi(X)), Y-X \rangle \geq -\kappa_g L\|Y-X\|^2.
\end{align}
By combining \eqref{l-1} and \eqref{l-2}, the lemma is proved.
\hfill$\blacksquare$

\subsection{Proof of Lemma \ref{P-lip}}
We first introduce a technical lemma.
\begin{lemma}
	\cite{davis2020} If a function $\phi$ is assumed to be $L$-Lipschitz continuous on $\mathcal{M}$ and $\rho$-weakly convex. Let $\beta \geq \rho+\frac{3L}{R}$. Fix a point $X\in \mathcal{M}$ and define the proximal point $\tilde{X}:=\arg \min_{Y \in \mathcal{M}}\phi(Y)+\frac{\beta}{2}\|X-Y\|^2$. Then  the following  holds.
	\begin{align}\label{3point}
	\phi(Y) - \phi(\tilde{X})\leq& \frac{\beta-\rho -\frac{3L}{R}}{2}\|Y-\tilde{X}\|^2 + \frac{\beta}{2}\|X-\tilde{X}\|^2\notag\\
	-&\frac{\beta}{2}\|Y-X\|^2,\quad \forall Y \in \mathcal{M}.
	\end{align}
\end{lemma}

$P_{\lambda f}(X)$ is single-valued can be easily derived since $f$ is weakly-convex function. 

Next, we prove the Lipschitz-type inequality. 
Denote $\tilde{X}:= P_{\lambda f}(X)$ and $\tilde{Y}:= P_{\lambda f}(Y)$. By the optimality of the proximal map, there exist $g_X \in \partial f(\tilde{X})$ and $p\in N_\mathcal{M}(\tilde{X})$ such that $X-\tilde{X} = \lambda(g_X+p)$. Similarly, there exist $g_Y \in \partial f(\tilde{Y})$ and $q\in N_\mathcal{M}(\tilde{Y})$ such that $Y-\tilde{Y} = \lambda(g_Y+q)$. Subsequently, since $f$ is $\rho$-weakly-convex, by using \eqref{3point} and \eqref{cur}, we have \begin{align}\label{1}
&\langle X-Y, \tilde{X}-\tilde{Y} \rangle \notag\\
 = &\langle \lambda(g_X+p)+\tilde{X} - \lambda(g_Y+q)-\tilde{Y} , \tilde{X}-\tilde{Y} \rangle\notag\\
=&\|\tilde{X}-\tilde{Y}\|^2+\lambda \langle g_X - g_Y, \tilde{X}-\tilde{Y} \rangle + \langle p-q, \tilde{X}-\tilde{Y} \rangle\\
\geq & (1-\lambda \rho)\|\tilde{X}-\tilde{Y}\|^2 -\lambda \kappa_g(\|p\|+\|q\|)\|\tilde{X}-\tilde{Y}\|^2.\notag
\end{align}
Since $f$ is $L$-Lipschitz continuous, by the definition of $P_{\lambda f}(X)$ we directly derive $\|X - \tilde{X}\| \leq 2L\lambda$. Based on this we can derive that $\|p\| \leq \frac{1}{\lambda}\|X-\tilde{X}\|+\|g_X\| \leq 3L$. Similarly, $\|q\|\leq 3L$. Plugging them into \eqref{1} yields \[\|\tilde{X}-\tilde{Y}\|\cdot\|X-Y\| \geq (1-\lambda(\rho+6\kappa_g L))\|\tilde{X}-\tilde{Y}\|^2.\] 
As $\lambda \in (0, \frac{1}{\rho+6\kappa_g L})$, Lemma \ref{P-lip} is proved.\hfill$\blacksquare$

\subsection{Distributed Blind Deconvolution on Sphere}
\begin{figure*}[htbp]
	\centering
	\setlength{\abovecaptionskip}{0cm}
	\subfloat[Ring Graph]{ \includegraphics[width=0.3\linewidth]{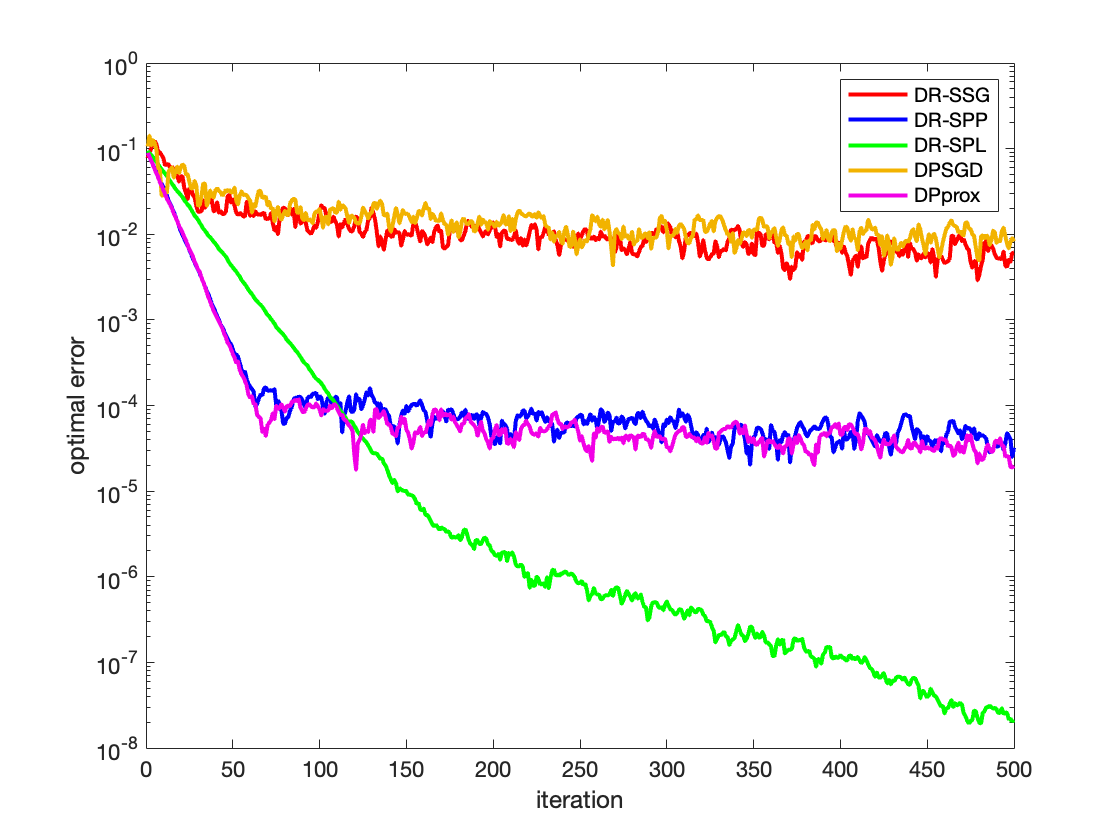} }
	\subfloat[ER Graph ($p=0.2$)]{ \includegraphics[width=0.3\linewidth]{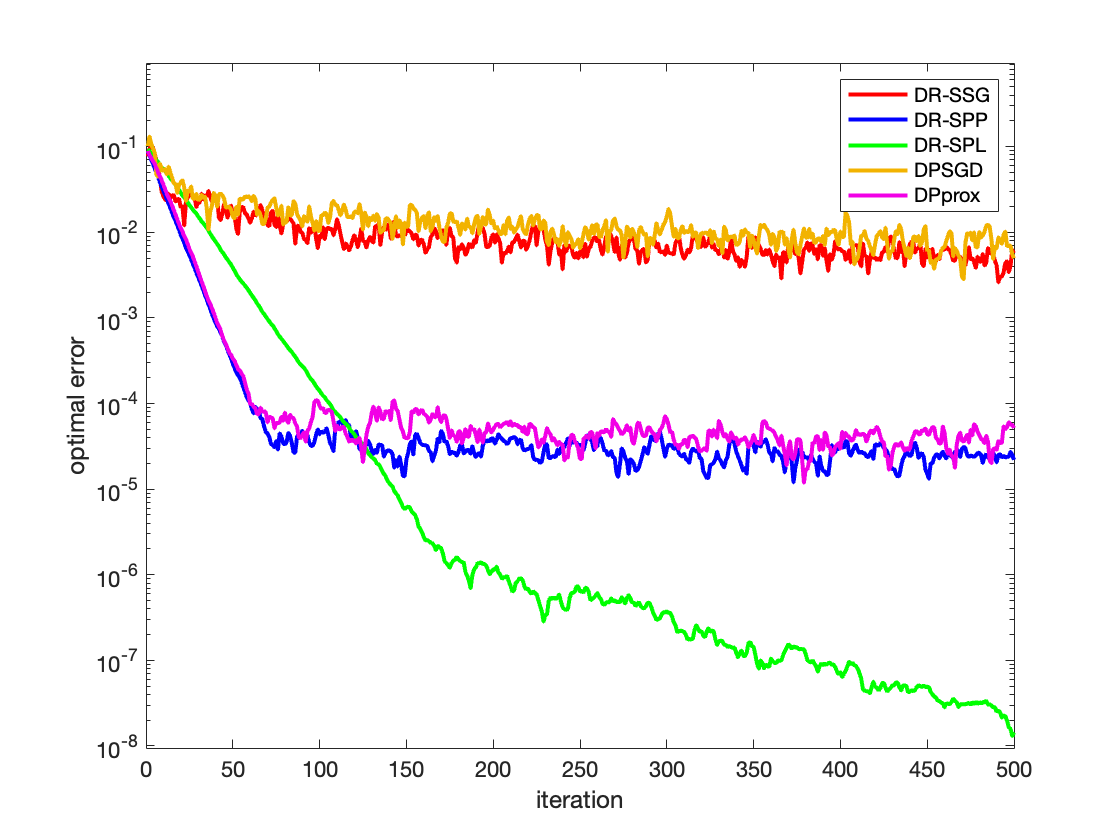} }
	\subfloat[Complete Graph]{ \includegraphics[width = 0.3\linewidth]{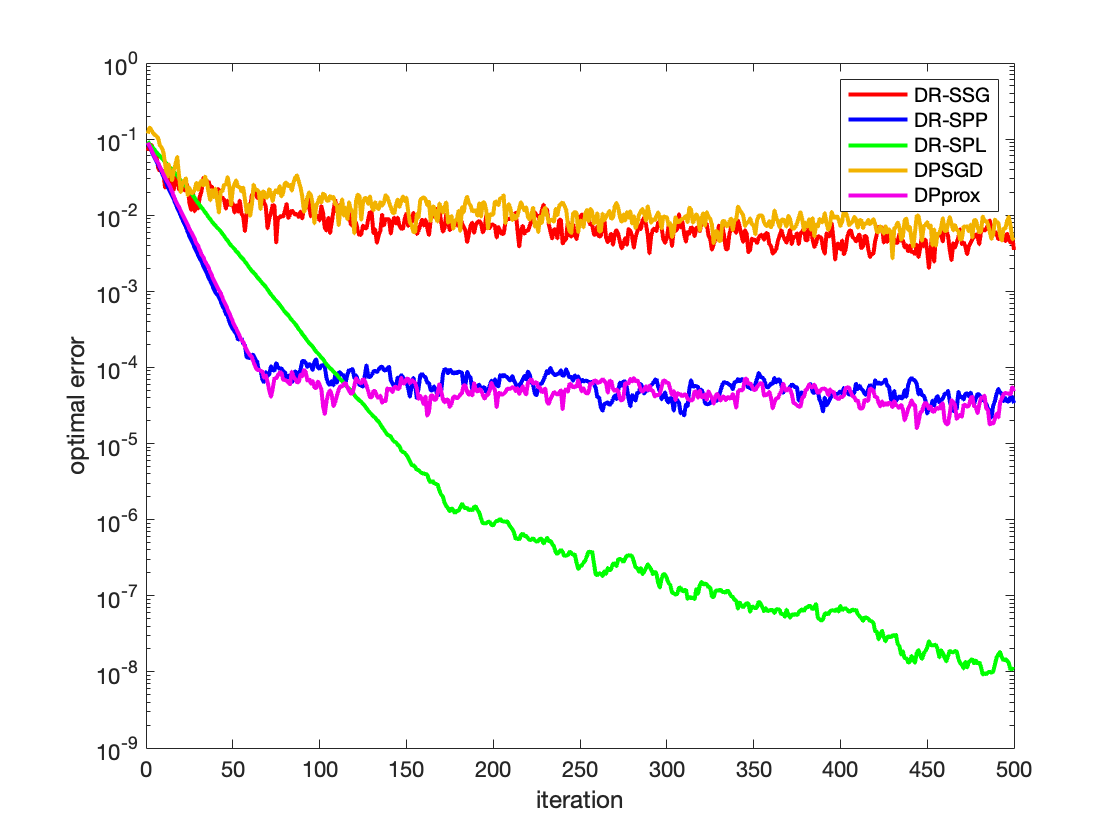} }
	\caption{The performance of the three proximal model-based algorithms, DPSGD, and DPprox under different communication graphs, with $\beta_0 = 0.15$ for DR-SSG and DPSGD, $\beta_0 = 0.1$ for DR-SPP and DPprox, and $\beta_0 = 0.5$ for DR-SPL.}
	\label{fig:bd}
\end{figure*}

Blind deconvolution problem aims at recovering a pair of signals $x \in S^{n-1} \subset \mathbb{R}^{n}$ and $y \in \mathbb{R}^{m}$ from their convolution $b = x \odot y$, where $\odot$ represent pairwise convolution. 
Denote $S^{n-1}:=\{x \in \mathbb{R}^n| x^\top x = 1\}$ as a $n$-dimensional sphere, which is an embedded submanifold of $\mathbb{R}^n$. Focusing on the real-valued case, the distributed problem can be modeled as 
$\min_{x \in S^{n-1},y \in \mathbb{R}^{m}}~\frac{1}{N}\sum_{i=1}^N E_{a_i,c_i}[|\langle a_i,x \rangle \langle c_i,y \rangle - b_i|],$ where $i$ is the index of agent, $a_i,c_i$ are known values obey certain distributions and $b_i$ denotes the convolution measurement.

Let the measurements be generated from the standard Gaussian distribution $a_i \sim N(0, I_{n\times n})$ and $c_i \sim N(0, I_{m\times m})$ where $n=10$ and $m=15$. We randomly generate the target signal $x^*$ on the unit sphere and $y^* \in \mathbb{R}^m$. Each node collects the convolution measurement modeled by $b_i = \langle a_i, x^* \rangle \langle c_i, y^* \rangle$.

Assume there are $N=20$ nodes that work together to solve the problem, and each node has a local copy of the convolution kernel $x_i$ and signal $y_i$  satisfying $x_1=\dots =x_N$ and $y_1=\dots = y_N$. We analyze the algorithms' performance under three different communication networks where the underlying graphs are ring graph, ER graph ($p=0.2$), and complete graph, respectively.  Weights of the adjacency matrix are constructed by the Metropolis rule ($\epsilon=0.5$)
$$w_{ij} = \left\{ \begin{array}{ll}
\frac{1}{\max \{|d_i|,|d_j| \}+\epsilon},&(i,j) \in \mathcal{E}\\
0,&(i,j) \notin \mathcal{E},~i\ne j\\
1-\sum_{j\ne i} w_{ij},&i=j
\end{array}  \right. $$ Let the multi-step consensus $t=1$, and the step size $\alpha = 1$. The step sizes satisfy $\beta_k=\frac{\sqrt{k+1}}{\beta_0}$. Initials are randomly generated and equal for all nodes. The optimal error is measured by the error between $(\hat x, \bar y)$ and $(x^*,y^*)$. We implement our proposed algorithms as follows.

{\bf Distributed Riemannian stochastic subgradient (DR-SSG).} Since the local cost function is expectation-valued, the distributed stochastic subgradient evaluate the local subgradient of each agent by randomly calculate $\partial_R(f_i(x_i; (a_i,c_i))) = \mathcal{P}_{T_{X_i}\mathcal{M}}[\partial f_i(x_i; (a_i,c_i))]$, where $(a_i,c_i)$ represent the data sample selected from their distribution.

{\bf Distributed Riemannian stochastic proximal point (DR-SPP).} The distributed stochastic proximal point algorithm needs to solve the local subproblem $\min_{v\in T_{x_{i,k}} S^{n-1},y}~ |\langle a_i,(x_{i,k}+v) \rangle \langle c_i,y \rangle - b_i| +\frac{\beta_k}{2}\|v\|^2+ \frac{\beta_k}{2}\|y-y_{i,k}\|^2$.

{\bf Distributed Riemannian stochastic proximal linear (DR-SPL).} Finally, the distributed proximal linear algorithm also requires to calculate a subproblem at each iteration to obtain the search direction:
$\min_{(v_1,v_2) \in T_{x_{i,k}}S^{n-1} \times \mathcal{R}^m }~|\langle a_i,x_{i,k} \rangle \langle c_i,y_{i,k} \rangle + \langle c_i,y_{i,k} \rangle \langle a_i,v_1 \rangle+ \langle a_i,x_{i,k} \rangle \langle c_i,v_2 \rangle-b |+\frac{\beta_k}{2} \|v_1\|^2+\frac{\beta_k}{2} \|v_2\|^2.$ 

The solution methods of the subproblems are provided in Appendix D.                                                                                                             

We compare the proposed algorithms with the stochastic version of \cite{deng2023decentralized} (DPSGD) and \cite{wang2024proxtrack} (DPprox). The empirical results are shown in Fig.\ref{fig:bd}. In all communication structures, our proposed DR-SPL exhibits superior long-term convergence characteristics finally achieving the highest precision, while its initial descent is more gradual compared to proximal point algorithms, namely DPprox [51]  and DR-SPP.  Our proposed DR-SPP has a considerable performance with DPprox [51], which achieves the fastest initial convergence during the first 150 iterations, but reaches a plateau thereafter, allowing DR-SPL to eventually outperform it. Moreover, DR-SSG has the slowest convergence rate and smallest accuracy, performing similarly to DPSGD [17].
 In general, the DR-SPL algorithm shows the best overall performance from both the convergence rate and the stability perspectives.

\end{document}